\documentclass[10pt]{amsart}%
\usepackage[T1]{fontenc}
\usepackage{amsmath,amssymb}
\usepackage{fullpage}
\usepackage{amsthm}
\usepackage{multirow}

\usepackage[nonewpage]{imakeidx}
\makeindex[columns=2, title=Index of notations]

\usepackage{color}
\usepackage{enumerate}
\usepackage{tikz}
\usepackage{ytableau}
\usepackage{pifont}
\usepackage{amsmath}
\usepackage{graphicx}
  \usepackage[all]{xy}
\usepackage{amsfonts}
\usepackage{amssymb}
\usepackage{ytableau}%

\providecommand{\U}[1]{\protect\rule{.1in}{.1in}}
\providecommand{\U}[1]{\protect\rule{.1in}{.1in}}
\providecommand{\U}[1]{\protect\rule{.1in}{.1in}}
\providecommand{\U}[1]{\protect\rule{.1in}{.1in}}
\providecommand{\U}[1]{\protect\rule{.1in}{.1in}}
\newcommand{\ulambda}{{\boldsymbol{\lambda}}}
\newcommand{\uLambda}{{\boldsymbol{\Lambda}}}

\newcommand{\umu}{{\boldsymbol{\mu}}}

\newtheorem{Th}{Theorem}[subsection]
\numberwithin{equation}{section}

\newtheorem{lemma}[Th]{Lemma}
\newtheorem{Cor}[Th]{Corollary}
\newtheorem{Prop}[Th]{Proposition}

\theoremstyle{remark}
\newtheorem{Rem}[Th]{Remark}{\rmfamily}
\theoremstyle{definition}
{\rmfamily}
\newtheorem{exa}[Th]{Example}{\rmfamily}

\def\C{\mathbb{C}}
\def\Z{\mathbb{Z}}

\newcommand{\Irr}{\operatorname{Irr}}

\def\Irr{{\mathrm{Irr}}}
\def\Q{{\mathbb Q}}
\def\N{{\mathbb N}}
\def\Z{{\mathbb Z}}

\def\C{{\mathbb C}}

\newcommand\blfootnote[1]{%
  \begingroup
  \renewcommand\thefootnote{}\footnote{#1}%
  \addtocounter{footnote}{-1}%
  \endgroup
}

\thanks{Research supported by the Hellenic Foundation for Research and Innovation (H.F.R.I.) under the Basic Research
Financing (Horizontal support for all Sciences), National Recovery and Resilience Plan (Greece 2.0), Project Number:
15659, Project Acronym: SYMATRAL. N.J. is 
 supported by  Agence Nationale de la
Recherche funding ANR CORTIPOM 21-CE40-001. The authors would like to thank Gunter Malle for his useful comments.}

\begin{document}

\title{Generalised hook lengths and Schur elements for Hecke algebras}

\author{Maria Chlouveraki}
\address{National and Kapodistrian University of Athens, Department of Mathematics, 
Panepistimioupolis 15784 Athens, Greece.}
\email{mchlouve@math.uoa.gr}

\author{Jean-Baptiste Gramain}
\address{Institute of Mathematics, University of Aberdeen,
King's College,
Fraser Noble Building,
Aberdeen AB24 3UE,
UK.}
\email{jbgramain@abdn.ac.uk}

\author{Nicolas Jacon}
\address{Universit\'{e} de Reims Champagne-Ardennes, UFR Sciences exactes et
naturelles. Laboratoire de Math\'{e}matiques UMR CNRS 9008. Moulin de la Housse BP
1039. 51100 REIMS. France.}
\email{nicolas.jacon@univ-reims.fr} 

\maketitle
\date{}
\blfootnote{\textup{2020} \textit{Mathematics Subject Classification}: \textup{05E10, 20C08,  20C20}} 
\begin{abstract}
We compare two generalisations of the notion of  hook lengths for partitions.
We apply this in the context of the modular representation theory of Ariki-Koike algebras. We show that the Schur element of a simple module 
 is divisible by the Schur element  of the associated (generalised) core. In the case of Hecke algebras of type $A$, we obtain an even stronger result: the Schur element of a simple module is equal to the product of the Schur element of its core and the Schur element of its quotient.
\end{abstract}

\section{Introduction}

The representation theory of the symmetric group $\mathfrak{S}_n$ over a field of characteristic $e$  and that of its Hecke algebra $\mathcal{H}_n(q)$ when the parameter $q$ is specialised to an $e$-th root of unity are  closely  connected. In particular, they can  both be  nicely described using the combinatorics of partitions and Young tableaux.  

Part of the information on the representations  of these algebras can be obtained by studying a fundamental object: the decomposition matrix. This matrix is a block diagonal matrix whose rows are indexed by the partitions of $n$.
These label the irreducible representations when the symmetric group algebra or the Hecke algebra are semisimple (for example, over $\C$ or when $q$ is an indeterminate, respectively). Then the irreducible representations of the symmetric group can be obtained from the ones of the Hecke algebra by taking $q=1$. Their dimension is given by the famous hook length formula, which states that the dimension of an irreducible representation labelled by a partition $\lambda$ is equal to $n!$ divided by the product of the hook lengths associated with $\lambda$.

 It turns out that the blocks of the decomposition matrix can be described using a classical combinatorial process on partitions.
   One can associate to each partition a pair consisting of a single partition, the $e$-core,  and an $e$-tuple of partitions, the $e$-quotient.    The 
Nakayama Conjecture (proved by Brauer and Robinson \cite{BR}) then asserts  that two partitions label rows with non-zero entries in the same block if and only if they have the same $e$-core. In particular,  if the partition is equal to its $e$-core (in which case we simply say that it is an $e$-core), then the block  is a  $1\times 1 $ identity matrix.

The  Hecke algebra of the symmetric group has a natural generalisation: the  Ariki-Koike algebra, also considered as the Hecke algebra of the complex reflection group $G(l,1,n)$. For $l=1$, this is the symmetric group $\mathfrak{S}_n=A_{n-1}$, and for $l=2$, this is the Weyl group $B_n$. 
The representation theory of Ariki--Koike algebras  has been extensively studied during the past decades and can be developed in the same spirit as in type $A$. In particular, the rows of the associated decomposition matrix are indexed by  the  set of $l$-partitions of $n$, that is, $l$-tuples of partitions whose sizes add up to $n$.  It is then natural, from a representation theoretic point of view but also for the sake of algebraic combinatorics, to ask for a generalisation of the notions of $e$-core, $e$-quotient and hook lengths. We currently know of two independent generalisations of the notion of hook lengths to multipartitions. 
One was suggested by the second author together with Bessenrodt and Olsson in a paper of 2012, \cite{BGO}. Using this, the authors were able to  show that the hook lengths of a partition always contains those of its $e$-core, giving a combinatorial  interpretation of a result by Malle and Navarro \cite[Theorem 9.1]{MN}.
The other generalisation was suggested and used by the first and third authors in another paper of 2012, \cite{CJ1}, in order to describe the Schur elements of Ariki--Koike algebras. It was only after it was reused in the more recent paper \cite{CJ2} that we realised that the two notions, even though different at a first glance, could be compared. This is the first aim of this paper.  Using the language of $l$-symbols, we are able to obtain an injection from the multiset of CJ-hook lengths to the multiset of BGO-hook lengths which preserves the absolute value of the hook length (Proposition \ref{inj}). In the equal charge case, this injection becomes a bijection (Corollary \ref{bij}).

Now, as far as the notion of $e$-core  (and $e$-quotients) is concerned, a first combinatorial generalisation was suggested also in \cite{BGO}, while a second one, adapted to the representation theory of Ariki--Koike algebras, was suggested by the third author and Lecouvey in \cite{JL}. In Section \ref{core}, we give a connection between the two, but it is  weaker than the one provided for generalised hook lengths. However, thanks to this connection, we are able to deduce (Proposition \ref{prop2.3}) that the CJ-hook lengths of the $e$-core of a  multipartion $\ulambda$ are contained in those of $\ulambda$ (or an extension thereof).

In the second part of the paper, we explore the consequences of our results on the representation theory of Ariki--Koike algebras. Ariki--Koike algebras are symmetric algebras, that is, they are endowed with a linear map which is a symmetrising trace; this is also true for group algebras of finite groups.
To each irreducible representation of a symmetric algebra, we can associate an element  of the (integral closure of) the ring over which the algebra is defined, the \emph{Schur element}. In the case of the group algebra of the symmetric group, the  Schur element associated with a partition $\lambda$ is the product of all hook lengths of $\lambda$. In the case of the Ariki--Koike algebra, the  Schur element associated with a multipartition $\ulambda$  is a Laurent polynomial in many indeterminates. There are three different descriptions of the Schur elements of Ariki--Koike algebras \cite{GIM, Mat, CJ1}, the latter of which uses the CJ-hook lengths.

 In this paper, we use our results on generalised hook lengths to prove two theorems on Schur elements. The first one concerns the factorisation of  Schur elements in type $A$. Theorem 
 \ref{mainTypeA} states that, for any $e$, the Schur element associated with a partition $\lambda$ is the product of  the Schur element of its $e$-core and  the Schur element of its $e$-quotient. 
 We obtain thus a connection between the representation theory of the Hecke algebra of type $A$ and that of the cyclotomic Hecke algebra of type $G(e,1,n)$. 
 Our second result, Proposition \ref{unmaintheorem}, concerns more generally the Schur elements of Ariki-Koike algebras: we show a divisibility property which can be regarded as as a generalisation of \cite[Proposition 3.8.3]{CJ2}. This property allows us to compare the Schur element of an arbitrary multipartition with that of its generalised core (in the sense of \cite{JL}). 

Finally, we should mention that many years before \cite{BGO} and \cite{CJ1}, Malle  defined hook lengths and cores for $l$-symbols with specific charges, which he used to obtain a connection between the degrees of a unipotent character and its corresponding cuspidal, in \cite{Mal}. Malle's hook lengths resemble the BGO-hook lengths, but, as it is also mentioned in \cite{BGO}, his are defined for equivalence classes of $l$-symbols and only include non-zero hook lengths. However, 
the results of the second part of our paper should be related to $d$--Howlett--Lehrer theory whenever we are in the setting of a finite group of Lie type or, more generally, a Spets (cf.~\cite{BMM}),
given that unipotent degrees correspond to certain ``spetsial'' specialisations of inverses of Schur elements multiplied by the Poincar\'e polynomial.

\section{Generalised hook lengths}
In this section, we introduce two notions which generalise the notions of hook and hook lengths for partitions to the case of multipartitions (and of their associated symbols). We then investigate the connection between these  objects.

\subsection{Symbols}\label{Symbols} Let $l\in \mathbb{Z}_{>0}$. 
 By definition,  a {\it $\beta$-set} of {\it charge}  $m$  is a collection $X=(a_1,\ldots,a_m)$ of strictly increasing elements in $\mathbb{N}$. 
An {\it $l$-symbol} is a collection of $l$ $\beta$-sets $X=(X_1,\ldots,X_l)$. 
The {\it multicharge} of the symbol is the $l$-tuple $(m_1,\ldots,m_l)$ where, for all $j=1,\ldots,l$, the integer $m_j$ is the charge of the $\beta$-set $X_j$. 
 
\begin{exa}
The 3-symbol $X=((0,2,4,6),(0,3,4),(0,2,5))$ has multicharge $(4,3,3)$.
\end{exa}
An $l$-symbol $X=(X_1,\ldots,X_l)$ can be conveniently represented using its abacus configuration. The $l$ $\beta$-sets $X_1,\ldots, \, X_l$ are represented by $l$ horizontal runners, labelled from bottom to top. Each runner is full of beads, numbered by the natural integers, and a bead numbered by 
$a\in \mathbb{N}=\mathbb{Z}_{\geq 0}$ is coloured black if and only if $a\in X_j$. Thus the  charge of $X_j$ is the number of black beads in the associated runner. For each $1 \leq j \leq l$, we also let $\Gamma_j= \N \setminus X_j$.

  \begin{exa}\label{firstabacus}
  If $l=3$ and $X=((0,2,4,6),(0,3,4),(0,2,5))$, then the abacus representation of $X$ is
\begin{center}
\begin{tikzpicture}[scale=0.5, bb/.style={draw,circle,fill,minimum size=2.5mm,inner sep=0pt,outer sep=0pt}, wb/.style={draw,circle,fill=white,minimum size=2.5mm,inner sep=0pt,outer sep=0pt}]
	
	\node [] at (11,-1) {20};
	\node [] at (10,-1) {19};
	\node [] at (9,-1) {18};
	\node [] at (8,-1) {17};
	\node [] at (7,-1) {16};
	\node [] at (6,-1) {15};
	\node [] at (5,-1) {14};
	\node [] at (4,-1) {13};
	\node [] at (3,-1) {12};
	\node [] at (2,-1) {11};
	\node [] at (1,-1) {10};
	\node [] at (0,-1) {9};
	\node [] at (-1,-1) {8};
	\node [] at (-2,-1) {7};
	\node [] at (-3,-1) {6};
	\node [] at (-4,-1) {5};
	\node [] at (-5,-1) {4};
	\node [] at (-6,-1) {3};
	\node [] at (-7,-1) {2};
	\node [] at (-8,-1) {1};
	\node [] at (-9,-1) {0};

	\node [wb] at (11,0) {};
	\node [wb] at (10,0) {};
	\node [wb] at (9,0) {};
	\node [wb] at (8,0) {};
	\node [wb] at (7,0) {};
	\node [wb] at (6,0) {};
	\node [wb] at (5,0) {};
	\node [wb] at (4,0) {};
	\node [wb] at (3,0) {};
	\node [wb] at (2,0) {};
	\node [wb] at (1,0) {};
	\node [wb] at (0,0) {};
	\node [wb] at (-1,0) {};
	\node [wb] at (-2,0) {};
	\node [bb] at (-3,0) {};
	\node [wb] at (-4,0) {};
	\node [bb] at (-5,0) {};
	\node [wb] at (-6,0) {};
	\node [bb] at (-7,0) {};
	\node [wb] at (-8,0) {};
	\node [bb] at (-9,0) {};
	
	\node [wb] at (11,1) {};
	\node [wb] at (10,1) {};
	\node [wb] at (9,1) {};
	\node [wb] at (8,1) {};
	\node [wb] at (7,1) {};
	\node [wb] at (6,1) {};
	\node [wb] at (5,1) {};
	\node [wb] at (4,1) {};
	\node [wb] at (3,1) {};
	\node [wb] at (2,1) {};
	\node [wb] at (1,1) {};
	\node [wb] at (0,1) {};
	\node [wb] at (-1,1) {};
	\node [wb] at (-2,1) {};
	\node [wb] at (-3,1) {};
	\node [wb] at (-4,1) {};
	\node [bb] at (-5,1) {};
	\node [bb] at (-6,1) {};
	\node [wb] at (-7,1) {};
	\node [wb] at (-8,1) {};
	\node [bb] at (-9,1) {};
	
	\node [wb] at (11,2) {};
	\node [wb] at (10,2) {};
	\node [wb] at (9,2) {};
	\node [wb] at (8,2) {};
	\node [wb] at (7,2) {};
	\node [wb] at (6,2) {};
	\node [wb] at (5,2) {};
	\node [wb] at (4,2) {};
	\node [wb] at (3,2) {};
	\node [wb] at (2,2) {};
	\node [wb] at (1,2) {};
	\node [wb] at (0,2) {};
	\node [wb] at (-1,2) {};
	\node [wb] at (-2,2) {};
	\node [wb] at (-3,2) {};
	\node [bb] at (-4,2) {};
	\node [wb] at (-5,2) {};
	\node [wb] at (-6,2) {};
	\node [bb] at (-7,2) {};
	\node [wb] at (-8,2) {};
	\node [bb] at (-9,2) {};
	\end{tikzpicture}
\end{center}
  \end{exa}

  \subsubsection{Generalised hooks} We have two types of generalisation of the notion of hook. 
A (general) {\it hook} in the $l$-symbol $X$
  is the datum of a quadruple $(a,b,i,j)$ with $a, \, b \in \N$ and $1 \leq i, \, j \leq l$, and where $a\in X_i$ and $b \in \Gamma_j$. 
 \begin{enumerate}
 \item A {\it BGO-hook}  is a hook $(a,b,i,j)$ such that $a \geq b$ and, if $a=b$, then $i>j$.\smallbreak  
 \item A {\it CJ-hook} is a hook $(a,b,i,j)$ such that  
 $$\sharp \{\gamma \in \Gamma_i \,|\,a> \gamma  \} > \sharp  \{\gamma \in \Gamma_j \,|\,b>\gamma \}.$$ 
  \end{enumerate}
BGO-hooks were defined in \cite{BGO}, and CJ-hooks in \cite{CJ1} (with the abacus interpretation given in \cite{CJ2}).

\begin{exa}
Let $X=((0,2,4,6),(0,3,4),(0,2,5))$. Then $(2,2,3,2)$ is a BGO-hook in $X$, but it is not a CJ-hook, whereas $(3,3,2,3)$ is a CJ-hook but not a BGO-hook. 

\end{exa}
 We will denote by \index{$\mathcal{H}^{BGO} (X)$} $\mathcal{H}^{BGO} (X)$ and \index{$\mathcal{H}^{CJ} (X)$}$\mathcal{H}^{CJ} (X)$ the sets of BGO-hooks and CJ-hooks in $X$ respectively. 
When $l=1$, the two notions coincide. Moreover, 
 for all $i=1,\ldots,l$, we have
$(a,b,i,i) \in \mathcal{H}^{BGO} (X)$ if and only if $(a,b,i,i) \in \mathcal{H}^{CJ} (X)$.
We obviously have
$$\mathcal{H}^{BGO} (X) = \bigcup_{1 \leq i<j \leq l} \mathcal{H}^{BGO} ((X_i,X_j)) \quad \text{and}\quad
 \mathcal{H}^{CJ} (X) = \bigcup_{1 \leq i<j \leq l} \mathcal{H}^{CJ}  ((X_i,X_j)).$$

Let now  $i,j \in \{1,\ldots, l\}$.
If $a \in X_i$, then we denote by \index{$ \mathcal{H}^{BGO}_{i,j}(a)$} $ \mathcal{H}^{BGO}_{i,j}(a)$ (resp. \index{$\mathcal{H}^{CJ}_{i,j}(a)$} $\mathcal{H}^{CJ}_{i,j}(a)$) the  set of all BGO-hooks (resp. CJ-hooks) of the form $(a,b,i,j)$. We have
$$\sharp \mathcal{H}^{CJ}_{i,j}(a)  = \sharp  \{\gamma \in \Gamma_i \,|\,a> \gamma  \}.$$
If $i \leq j$, then
$$\sharp \mathcal{H}^{BGO}_{i,j}(a)  = \sharp  \{\gamma \in \Gamma_j \,|\,a> \gamma  \}.$$
If $i>j$, 
then
$$\sharp \mathcal{H}^{BGO}_{i,j}(a)  = \sharp  \{\gamma \in \Gamma_j \,|\,a \geq  \gamma  \}.$$

\begin{Prop}\label{Number}
Let $X=(X_1,\ldots,X_l)$ be an $l$-symbol with multicharge $(m_1,\ldots,m_l)$.  
If $m_i=m_j$ for some $1 \leq i<j \leq l$, then 
$$\sharp \mathcal{H}^{BGO}  ((X_i,X_j))=\sharp \mathcal{H}^{CJ}  ((X_i,X_j)).$$
\end{Prop}
\begin{proof}
Let $X_i=(a_1,\ldots,a_m)$ and $X_j=(b_1,\ldots,b_m)$, where $m=m_i=m_j$. 
If $a_m \leq b_m$, let $k \in \{1,\ldots,m\}$ be minimal such that $a_m \leq b_k$, and set $a:=a_m$ and $b:=b_k$.
If $a_m > b_m$, let $k \in \{1,\ldots,m\}$ be minimal such that $a_k > b_m$, and set $a:=a_k$ and $b:=b_m$.
In both cases, we have
$$\begin{array}{rcl}
\sharp  \mathcal{H}^{CJ}_{i,j}(a) + \sharp \mathcal{H}^{CJ}_{j,i}(b) &=& 
\sharp  \{\gamma \in \Gamma_i \,|\,a > \gamma  \} + \sharp  \{\gamma \in \Gamma_j \,|\,b > \gamma  \} \\&=&
\sharp  \{\gamma \in \Gamma_i \,|\,{\rm min}(a,b) > \gamma  \} + \sharp  \{\gamma \in \Gamma_j \,|\,{\rm min}(a,b)  > \gamma  \} +
|a-b|
\end{array}$$
and 
$$\begin{array}{rcl}
\sharp  \mathcal{H}^{BGO}_{i,j}(a) + \sharp \mathcal{H}^{BGO}_{j,i}(b) &=& 
\sharp  \{\gamma \in \Gamma_j \,|\,a> \gamma  \} + \sharp  \{\gamma \in \Gamma_i \,|\,b \geq  \gamma  \}\\&=&
\sharp  \{\gamma \in \Gamma_j \,|\,{\rm min}(a,b)> \gamma  \} +\sharp  \{\gamma \in \Gamma_i \,|\,{\rm min}(a,b)> \gamma  \} +
|a-b|.
\end{array}$$
So it is enough to show that 
$$\sharp \mathcal{H}^{BGO}  ((X_i,X_j)) \setminus \left(\mathcal{H}^{BGO}_{i,j}(a) \cup \mathcal{H}^{BGO}_{j,i}(b)\right)=\sharp \mathcal{H}^{CJ}  ((X_i,X_j))\setminus \left(\mathcal{H}^{CJ}_{i,j}(a) \cup \mathcal{H}^{CJ}_{j,i}(b)\right).$$

We go on in a similar way until we have paired off each element of $X_i$ with an element of $X_j$ so that the total number of BGO-hooks  starting with these two elements is equal to the total number of CJ-hooks starting with these two elements.
   \end{proof}

\begin{Cor}\label{Number2}
Let $X=(X_1,\ldots,X_l)$ be an $l$-symbol with multicharge $(m_1,\ldots,m_l)$. If $m_1=\ldots=m_l$, then 
$$\sharp \mathcal{H}^{BGO} (X)=\sharp \mathcal{H}^{CJ} (X).$$
\end{Cor}

Now the {\it hook length} associated with the hook $(a,b,i,j)$ is \index{$\mathfrak{hl}(a,b,i,j)$} $\mathfrak{hl}(a,b,i,j):=a-b$. As the set of hooks are different in both cases in general, the associated hook lengths will be denoted by  \index{$\mathcal{HL}^{BGO} (X)$} $\mathcal{HL}^{BGO} (X)$ and \index{$\mathcal{HL}^{BGO} (X)$}$\mathcal{HL}^{CJ} (X)$. Note that the elements of the former are all non-negative integers whereas this is not the case in general for the latter.  The following lemma establishes a connection between CJ-hooks of the same length, which will be useful in the proof of the subsequent proposition.

\begin{lemma}\label{sim}
Let $X=(X_1,\ldots,X_l)$ be an $l$-symbol.
Let $h \in \Z$, and let $(a+h,a,i,j)$, $(a'+h,a',i,j)$ be hooks in $X$ with $a<a'$. Assume that, for all $k$ with $a<k<a'$, we have
either $(k+h,k) \in X_i \times X_j$ or $(k+h,k) \in \Gamma_i \times \Gamma_j$. Then  $(a+h,a,i,j)$ is a $CJ$-hook if and only if $(a'+h,a',i,j)$ is a $CJ$-hook.
\end{lemma}

\begin{proof}
Set $r:=\sharp \{a<k<a'\,|\,(k+h,k) \in \Gamma_i \times \Gamma_j\}$.
Then $\sharp \{\gamma \in \Gamma_j \,|\,a'> \gamma  \} = r+ \sharp \{\gamma \in \Gamma_j\,|\,a> \gamma  \}$ 
and $\sharp  \{\gamma \in \Gamma_i \,|\,a'+h>\gamma \}=r+
 \sharp  \{\gamma \in \Gamma_i \,|\,a+h>\gamma \}$. Hence,
 $$\sharp  \{\gamma \in \Gamma_i \,|\,a'+h>\gamma \} - \sharp \{\gamma \in \Gamma_j \,|\,a'> \gamma  \} 
 = \sharp  \{\gamma \in \Gamma_i \,|\,a+h>\gamma \}-\sharp \{\gamma \in \Gamma_j\,|\,a> \gamma  \}.$$ 
\end{proof}

\begin{Prop}\label{inj} Let $X=(X_1,\ldots,X_l)$ be an $l$-symbol. 
There exists an injective map 
$f: \mathcal{H}^{CJ} (X) \rightarrow  \mathcal{H}^{BGO}(X)$ such that, if $f(a,b,i,j)=(a',b',i',j')$,
then $|a-b|=a'-b'$  and $\{i,j\}=\{i',j'\}$. More specifically, if $(a,b,i,j) \notin \mathcal{H}^{BGO}(X)$, then $(a',b',i',j') \notin \mathcal{H}^{CJ} (X)$ and $(i',j')=(j,i)$, while the restriction of $f$ to $\mathcal{H}^{CJ} (X) \cap \mathcal{H}^{BGO}(X)$ is the identity map. 
\end{Prop}

\begin{proof}
First of all, if $(a,b,i,j) \in \mathcal{H}^{CJ} (X) \cap \mathcal{H}^{BGO}(X)$, then we set $f(a,b,i,j):=(a,b,i,j)$. Therefore, it is enough to show that there exists an injective map $$f: \mathcal{H}^{CJ} (X) \setminus  \mathcal{H}^{BGO}(X) \rightarrow  \mathcal{H}^{BGO}(X) \setminus  \mathcal{H}^{CJ}(X)$$ such that,  
if $f(a,b,i,j)=(a',b',i',j')$, 
then $|a-b|=b-a=a'-b'$ and $(i',j')=(j,i)$.  Note that a CJ-hook $(a,b,i,j)$ is not a BGO-hook if and only if $a-b<0$ or $a=b$ and $i<j$.

If $i,j \in \{1,\ldots,l\}$ with $i=j$, then $(a,b,i,j) \in \mathcal{H}^{CJ} (X) \cap \mathcal{H}^{BGO}(X)$, so there is nothing to do.

Now let $i,j \in \{1,\ldots,l\}$ with $i<j$. Let
$h \in \Z_{\geq 0}$ and let $a_1 < \cdots < a_r$ be all the elements of $X_i$ such that 
$(a_k,a_k+h,i,j) \in \mathcal{H}^{CJ}(X)$ ($1 \leq k \leq r$). These hooks have length $-h$, and so $(a_k,a_k+h,i,j) \notin \mathcal{H}^{BGO}(X)$ for $1 \leq k \leq r$.

We start with $(a_1,a_1+h,i,j)$. By definition, 
$$\sharp \{\gamma \in \Gamma_i \,|\,a_1> \gamma  \} > \sharp \{\gamma \in \Gamma_j \,|\,a_1+h> \gamma  \} \geq 0.$$
Let $c_1> \cdots > c_s$ be the elements of  $\{\gamma \in \Gamma_i \,|\,a_1> \gamma  \}$. 
If $c_1+h \in \Gamma_j$, then
$$\sharp \{\gamma \in \Gamma_i \,|\,a_1> \gamma  \} > \sharp \{\gamma \in \Gamma_j \,|\,a_1+h> \gamma  \} \geq 1,$$
and so $s >1$.  Now, if $c_2+h \in \Gamma_j$, then
$$\sharp \{\gamma \in \Gamma_i \,|\,a_1> \gamma  \} > \sharp \{\gamma \in \Gamma_j \,|\,a_1+h> \gamma  \} \geq 2,$$
and so $s >2$, and so on. Since $s$ is finite, this procedure must stop, and so there exists $t \in \{1,\ldots,s\}$ minimal with respect to the property $c_t+h \in X_j$. Then $(c_t+h,c_t,j,i) \in \mathcal{H}^{BGO}(X)$ and has hook length equal to $h$. 

We will now show that $(c_t+h,c_t,j,i) \notin \mathcal{H}^{CJ}(X)$. If not, then, by definition, 
$$\sharp \{\gamma \in \Gamma_j \,|\,c_t+h> \gamma  \} > \sharp \{\gamma \in \Gamma_i \,|\,c_t> \gamma  \} \geq 0.$$
 Set $u:= \sharp \{\gamma \in \Gamma_j \,|\,c_t+h> \gamma  \} $. 
  We have $u > s-t$. By Lemma \ref{sim} and because of the minimality of $a_1$, we have that 
  $\sharp \{\gamma \in \Gamma_j \,|\,a_1+h> \gamma>c_t+h  \}=t-1$. Hence,
  $\sharp \{\gamma \in \Gamma_j \,|\,a_1+h> \gamma\}=t-1+u$. Since $s > \sharp \{\gamma \in \Gamma_j \,|\,a_1+h> \gamma\}$, we deduce that $s-t > u-1$, which yields a contradiction. Therefore, $(c_t+h,c_t,j,i) \notin \mathcal{H}^{CJ}(X)$ and we set $f(a_1,a_1+h,i,j):=(c_t+h,c_t,j,i)$. 
  
  We now move on to $(a_2,a_2+h,i,j)$. By definition, 
$$\sharp \{\gamma \in \Gamma_i \,|\,a_2> \gamma  \} > \sharp \{\gamma \in \Gamma_j \,|\,a_2+h> \gamma  \} >
\sharp \{\gamma \in \Gamma_j \,|\,a_1+h> \gamma  \}  \geq 0.$$
Let $d_1> \cdots > d_v$ be the elements of  $\{\gamma \in \Gamma_i \,|\,a_2> \gamma  \} \setminus \{c_t\}$. 
  We have $v>0$. Using the same argument as before, if we take $w \in \{1,\ldots,v\}$ minimal with respect to the property $d_w+h \in X_j$, then $(d_w+h,d_w,j,i) \in \mathcal{H}^{BGO}(X) \setminus \mathcal{H}^{CJ}(X) $ and 
  we set $f(a_2,a_2+h,i,j):=(d_w+h,d_w,j,i)$.
  
  We continue like this for $a_3,\ldots,a_r$. 
  
Now let $i,j \in \{1,\ldots,l\}$ with $i>j$. Let
$h \in \Z_{> 0}$ and let $b_1 < \cdots < b_m$ be all the elements of $X_i$ such that 
$(b_k,b_k+h,i,j) \in \mathcal{H}^{CJ}(X)$ ($1 \leq k \leq m$). We define $f$ for these elements exactly as in the previous case.

Because of the way we defined $f$, it is clearly injective. 
\end{proof}

\begin{Cor}\label{bij}
Let $X=(X_1,\ldots,X_l)$ be an $l$-symbol  with multicharge $(m_1,\ldots,m_l)$. If $m_1=\ldots=m_l$, then the map $f$ is a bijection.
\end{Cor}
\begin{proof}
The result follows from Corollary \ref{Number2} and Proposition \ref{inj}. 
\end{proof}

  \begin{exa}\label{firstabacus} The map $f$ is not bijective in general. 
  Let $l=2$ and $X=((0,5),(0,1,4)).$
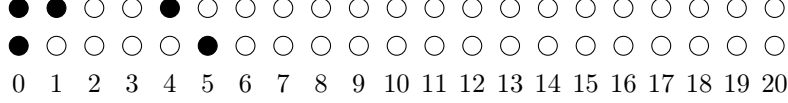
\begin{center}
\begin{tikzpicture}[scale=0.5, bb/.style={draw,circle,fill,minimum size=2.5mm,inner sep=0pt,outer sep=0pt}, wb/.style={draw,circle,fill=white,minimum size=2.5mm,inner sep=0pt,outer sep=0pt}]
	
	\node [] at (11,-1) {20};
	\node [] at (10,-1) {19};
	\node [] at (9,-1) {18};
	\node [] at (8,-1) {17};
	\node [] at (7,-1) {16};
	\node [] at (6,-1) {15};
	\node [] at (5,-1) {14};
	\node [] at (4,-1) {13};
	\node [] at (3,-1) {12};
	\node [] at (2,-1) {11};
	\node [] at (1,-1) {10};
	\node [] at (0,-1) {9};
	\node [] at (-1,-1) {8};
	\node [] at (-2,-1) {7};
	\node [] at (-3,-1) {6};
	\node [] at (-4,-1) {5};
	\node [] at (-5,-1) {4};
	\node [] at (-6,-1) {3};
	\node [] at (-7,-1) {2};
	\node [] at (-8,-1) {1};
	\node [] at (-9,-1) {0};

	\node [wb] at (11,0) {};
	\node [wb] at (10,0) {};
	\node [wb] at (9,0) {};
	\node [wb] at (8,0) {};
	\node [wb] at (7,0) {};
	\node [wb] at (6,0) {};
	\node [wb] at (5,0) {};
	\node [wb] at (4,0) {};
	\node [wb] at (3,0) {};
	\node [wb] at (2,0) {};
	\node [wb] at (1,0) {};
	\node [wb] at (0,0) {};
	\node [wb] at (-1,0) {};
	\node [wb] at (-2,0) {};
	\node [wb] at (-3,0) {};
	\node [bb] at (-4,0) {};
	\node [wb] at (-5,0) {};
	\node [wb] at (-6,0) {};
	\node [wb] at (-7,0) {};
	\node [wb] at (-8,0) {};
	\node [bb] at (-9,0) {};
	
	\node [wb] at (11,1) {};
	\node [wb] at (10,1) {};
	\node [wb] at (9,1) {};
	\node [wb] at (8,1) {};
	\node [wb] at (7,1) {};
	\node [wb] at (6,1) {};
	\node [wb] at (5,1) {};
	\node [wb] at (4,1) {};
	\node [wb] at (3,1) {};
	\node [wb] at (2,1) {};
	\node [wb] at (1,1) {};
	\node [wb] at (0,1) {};
	\node [wb] at (-1,1) {};
	\node [wb] at (-2,1) {};
	\node [wb] at (-3,1) {};
	\node [wb] at (-4,1) {};
	\node [bb] at (-5,1) {};
	\node [wb] at (-6,1) {};
	\node [wb] at (-7,1) {};
	\node [bb] at (-8,1) {};
	\node [bb] at (-9,1) {};
	
	\end{tikzpicture}
\end{center}
Then $\mathcal{HL}^{BGO} (X)=\{1,1,2,2,3,4 \,|\,0,0,1,2,3,2,3\}$ and $\mathcal{HL}^{CJ} (X)=\{1,1,2,2,3,4 \,|\,-1,0,2,3,2,3\}$ (from now on,  for the convenience of the reader, we separate by a ``|'' the hook lengths associated with hooks in the same component from the hook lengths associated with hooks in different components).
  \end{exa}

      \begin{exa}\label{firstabacus} The map $f$  can be bijective even when the charges are unequal. 
  Let $l=2$ and $X=((0,1,2,5),(0,3,4)).$
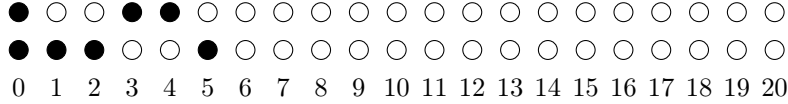
\begin{center}
\begin{tikzpicture}[scale=0.5, bb/.style={draw,circle,fill,minimum size=2.5mm,inner sep=0pt,outer sep=0pt}, wb/.style={draw,circle,fill=white,minimum size=2.5mm,inner sep=0pt,outer sep=0pt}]
	
	\node [] at (11,-1) {20};
	\node [] at (10,-1) {19};
	\node [] at (9,-1) {18};
	\node [] at (8,-1) {17};
	\node [] at (7,-1) {16};
	\node [] at (6,-1) {15};
	\node [] at (5,-1) {14};
	\node [] at (4,-1) {13};
	\node [] at (3,-1) {12};
	\node [] at (2,-1) {11};
	\node [] at (1,-1) {10};
	\node [] at (0,-1) {9};
	\node [] at (-1,-1) {8};
	\node [] at (-2,-1) {7};
	\node [] at (-3,-1) {6};
	\node [] at (-4,-1) {5};
	\node [] at (-5,-1) {4};
	\node [] at (-6,-1) {3};
	\node [] at (-7,-1) {2};
	\node [] at (-8,-1) {1};
	\node [] at (-9,-1) {0};

	\node [wb] at (11,0) {};
	\node [wb] at (10,0) {};
	\node [wb] at (9,0) {};
	\node [wb] at (8,0) {};
	\node [wb] at (7,0) {};
	\node [wb] at (6,0) {};
	\node [wb] at (5,0) {};
	\node [wb] at (4,0) {};
	\node [wb] at (3,0) {};
	\node [wb] at (2,0) {};
	\node [wb] at (1,0) {};
	\node [wb] at (0,0) {};
	\node [wb] at (-1,0) {};
	\node [wb] at (-2,0) {};
	\node [wb] at (-3,0) {};
	\node [bb] at (-4,0) {};
	\node [wb] at (-5,0) {};
	\node [wb] at (-6,0) {};
	\node [bb] at (-7,0) {};
	\node [bb] at (-8,0) {};
	\node [bb] at (-9,0) {};
	
	\node [wb] at (11,1) {};
	\node [wb] at (10,1) {};
	\node [wb] at (9,1) {};
	\node [wb] at (8,1) {};
	\node [wb] at (7,1) {};
	\node [wb] at (6,1) {};
	\node [wb] at (5,1) {};
	\node [wb] at (4,1) {};
	\node [wb] at (3,1) {};
	\node [wb] at (2,1) {};
	\node [wb] at (1,1) {};
	\node [wb] at (0,1) {};
	\node [wb] at (-1,1) {};
	\node [wb] at (-2,1) {};
	\node [wb] at (-3,1) {};
	\node [wb] at (-4,1) {};
	\node [bb] at (-5,1) {};
	\node [bb] at (-6,1) {};
	\node [wb] at (-7,1) {};
	\node [wb] at (-8,1) {};
	\node [bb] at (-9,1) {};
	
	\end{tikzpicture}
\end{center}
Then $\mathcal{HL}^{BGO} (X)=\{1,1,2,2,2,3\,|\,
0,0,1,1,3,4\}$ and $\mathcal{HL}^{CJ} (X)=\{1,1,2,2,2,3\,|\,-1,0,0,1,3,4\}$.
  \end{exa}

\subsubsection{Scaled hooks}

Let $k \in \mathbb{Z}_{>0}$.
If $X=(a_1,\ldots,a_m)$ is a $\beta$-set, then we denote by \index{$kX$ (where $X$ is a $\beta$-set or $l$-symbol)} $kX$ the $\beta$-set $(ka_1,\ldots,ka_m)$, and if $X=(X_1,\ldots,X_l)$ is an $l$-symbol, then we denote by $kX$ the $l$-symbol $(kX_1,\ldots,kX_l)$. 

Obviously, $(a,b,i,j)$ is a hook of $X$ if and only if $(ka,kb,i,j)$ is a hook of $kX$. Moreover, $(a,b,i,j)$ is a BGO-hook of $X$ if and only if $(ka,kb,i,j)$ is a BGO-hook of $kX$. The same does not hold for CJ-hooks in general (even though, by using the definition of CJ-hooks, we can show that it is true if $a \geq b$):

\begin{exa}\label{veryfirstabacus} 
  Let $l=2$ and $X=((0,2,3),(0,1,2)).$
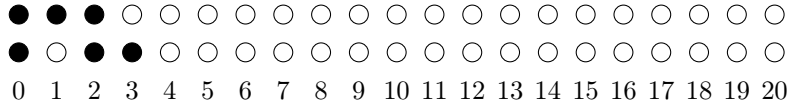
\begin{center}
\begin{tikzpicture}[scale=0.5, bb/.style={draw,circle,fill,minimum size=2.5mm,inner sep=0pt,outer sep=0pt}, wb/.style={draw,circle,fill=white,minimum size=2.5mm,inner sep=0pt,outer sep=0pt}]
	
	\node [] at (11,-1) {20};
	\node [] at (10,-1) {19};
	\node [] at (9,-1) {18};
	\node [] at (8,-1) {17};
	\node [] at (7,-1) {16};
	\node [] at (6,-1) {15};
	\node [] at (5,-1) {14};
	\node [] at (4,-1) {13};
	\node [] at (3,-1) {12};
	\node [] at (2,-1) {11};
	\node [] at (1,-1) {10};
	\node [] at (0,-1) {9};
	\node [] at (-1,-1) {8};
	\node [] at (-2,-1) {7};
	\node [] at (-3,-1) {6};
	\node [] at (-4,-1) {5};
	\node [] at (-5,-1) {4};
	\node [] at (-6,-1) {3};
	\node [] at (-7,-1) {2};
	\node [] at (-8,-1) {1};
	\node [] at (-9,-1) {0};

	\node [wb] at (11,0) {};
	\node [wb] at (10,0) {};
	\node [wb] at (9,0) {};
	\node [wb] at (8,0) {};
	\node [wb] at (7,0) {};
	\node [wb] at (6,0) {};
	\node [wb] at (5,0) {};
	\node [wb] at (4,0) {};
	\node [wb] at (3,0) {};
	\node [wb] at (2,0) {};
	\node [wb] at (1,0) {};
	\node [wb] at (0,0) {};
	\node [wb] at (-1,0) {};
	\node [wb] at (-2,0) {};
	\node [wb] at (-3,0) {};
	\node [wb] at (-4,0) {};
	\node [wb] at (-5,0) {};
	\node [bb] at (-6,0) {};
	\node [bb] at (-7,0) {};
	\node [wb] at (-8,0) {};
	\node [bb] at (-9,0) {};
	
	\node [wb] at (11,1) {};
	\node [wb] at (10,1) {};
	\node [wb] at (9,1) {};
	\node [wb] at (8,1) {};
	\node [wb] at (7,1) {};
	\node [wb] at (6,1) {};
	\node [wb] at (5,1) {};
	\node [wb] at (4,1) {};
	\node [wb] at (3,1) {};
	\node [wb] at (2,1) {};
	\node [wb] at (1,1) {};
	\node [wb] at (0,1) {};
	\node [wb] at (-1,1) {};
	\node [wb] at (-2,1) {};
	\node [wb] at (-3,1) {};
	\node [wb] at (-4,1) {};
	\node [wb] at (-5,1) {};
	\node [wb] at (-6,1) {};
	\node [bb] at (-7,1) {};
	\node [bb] at (-8,1) {};
	\node [bb] at (-9,1) {};
	
	\end{tikzpicture}
\end{center}
Then $\mathcal{HL}^{BGO} (X)=\{1,2 \,|\,0,1\}$ and $\mathcal{HL}^{CJ} (X)=\{1,2 \,|\,0,-1\}$. 
More specifically, $(3,3,1,2)$ and $(2,3,1,2)$ are CJ-hooks.
Now, for $k=2$, we have $kX=((0,4,6),(0,2,4))$.
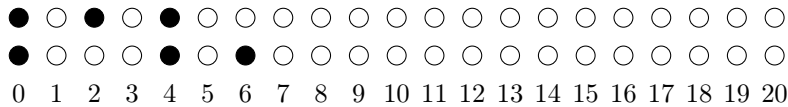
\begin{center}
\begin{tikzpicture}[scale=0.5, bb/.style={draw,circle,fill,minimum size=2.5mm,inner sep=0pt,outer sep=0pt}, wb/.style={draw,circle,fill=white,minimum size=2.5mm,inner sep=0pt,outer sep=0pt}]
	
	\node [] at (11,-1) {20};
	\node [] at (10,-1) {19};
	\node [] at (9,-1) {18};
	\node [] at (8,-1) {17};
	\node [] at (7,-1) {16};
	\node [] at (6,-1) {15};
	\node [] at (5,-1) {14};
	\node [] at (4,-1) {13};
	\node [] at (3,-1) {12};
	\node [] at (2,-1) {11};
	\node [] at (1,-1) {10};
	\node [] at (0,-1) {9};
	\node [] at (-1,-1) {8};
	\node [] at (-2,-1) {7};
	\node [] at (-3,-1) {6};
	\node [] at (-4,-1) {5};
	\node [] at (-5,-1) {4};
	\node [] at (-6,-1) {3};
	\node [] at (-7,-1) {2};
	\node [] at (-8,-1) {1};
	\node [] at (-9,-1) {0};

	\node [wb] at (11,0) {};
	\node [wb] at (10,0) {};
	\node [wb] at (9,0) {};
	\node [wb] at (8,0) {};
	\node [wb] at (7,0) {};
	\node [wb] at (6,0) {};
	\node [wb] at (5,0) {};
	\node [wb] at (4,0) {};
	\node [wb] at (3,0) {};
	\node [wb] at (2,0) {};
	\node [wb] at (1,0) {};
	\node [wb] at (0,0) {};
	\node [wb] at (-1,0) {};
	\node [wb] at (-2,0) {};
	\node [bb] at (-3,0) {};
	\node [wb] at (-4,0) {};
	\node [bb] at (-5,0) {};
	\node [wb] at (-6,0) {};
	\node [wb] at (-7,0) {};
	\node [wb] at (-8,0) {};
	\node [bb] at (-9,0) {};
	
	\node [wb] at (11,1) {};
	\node [wb] at (10,1) {};
	\node [wb] at (9,1) {};
	\node [wb] at (8,1) {};
	\node [wb] at (7,1) {};
	\node [wb] at (6,1) {};
	\node [wb] at (5,1) {};
	\node [wb] at (4,1) {};
	\node [wb] at (3,1) {};
	\node [wb] at (2,1) {};
	\node [wb] at (1,1) {};
	\node [wb] at (0,1) {};
	\node [wb] at (-1,1) {};
	\node [wb] at (-2,1) {};
	\node [wb] at (-3,1) {};
	\node [wb] at (-4,1) {};
	\node [bb] at (-5,1) {};
	\node [wb] at (-6,1) {};
	\node [bb] at (-7,1) {};
	\node [wb] at (-8,1) {};
	\node [bb] at (-9,1) {};
	
	\end{tikzpicture}
\end{center}
We have $\mathcal{HL}^{BGO} (kX)=\{1,1,1,1,2,3,3,3,4,5 \,|\,1,3,1,3,5,0,1,1,2,3\}$ and $\mathcal{HL}^{CJ} (kX)=\{1,1,1,1,2,3,3,3,4,5 \,|\,$ $-1,1,3,0,1,3,5,1,2,3\}$. 
Here, $(6,6,1,2)$ is a CJ-hook, but $(4,6,1,2)$ is not. However, we now have a new CJ-hook with hook length equal to $2$, the hook $(4,2,2,1)$. This is to be expected by Corollary \ref{bij}.
  \end{exa}

Set \index{$\mathcal{H}_k^*(X)$  } $\mathcal{H}_k^*(X):=\{h \in \mathcal{H}^*(kX) \,|\, \mathfrak{hl}(h) \equiv 0 \,{\rm mod}\,k\}$ where $* \in \{BGO,CJ\}$ (for $k=1$,  $\mathcal{H}_1^*(X)= \mathcal{H}^*(X)$). By definition of $kX$, the elements of $\mathcal{H}_k^*(X)$ are of the form $(ka,kb,i,j)$ for some $(a,b) \in X_i \times \Gamma_j$.
Following the discussion in the beginning of the subsection we have
\begin{equation}\label{maybe}
 \mathcal{H}^{BGO}_k(X) = \{(ka,kb,i,j)\,|\,(a,b,i,j) \in \mathcal{H}^{BGO}(X)\}.
 \end{equation}
Therefore, there exists a bijection
\begin{equation}\label{midbij}
 \mathcal{H}^{BGO}(X) \longrightarrow \mathcal{H}^{BGO}_k(X) , (a,b,i,j) \longmapsto (ka,kb,i,j).
 \end{equation}
 
Now, if $X=(X_1,\ldots,X_l)$ is an $l$-symbol  with multicharge $(m_1,\ldots,m_l)$, and $m_1=\ldots=m_l$, then Corollary
\ref{bij} yields the following bijections (since the map $f$ preserves the absolute value of hook lengths):
\begin{equation}\label{notmidbij2}
\mathcal{H}^{CJ}(X) \longleftrightarrow  \mathcal{H}^{BGO}(X)
\quad \text{ and } \quad \mathcal{H}^{BGO}_k(X) \longleftrightarrow  \mathcal{H}^{CJ}_k(X).
\end{equation}
Combining these with \eqref{midbij}, we obtain a bijection
\begin{equation}\label{notmidbij}
 \mathcal{H}^{CJ}(X) \longrightarrow  \mathcal{H}^{CJ}_k(X), (a,b,i,j) \longmapsto (ka',kb',i',j')
 \end{equation}
 such that $|a-b|=|a'-b'|$ and $\{i,j\}=\{i',j'\}$. At the level of the corresponding hook lengths, we have bijections:
 \begin{equation}\label{schl}
 \begin{array}{ccccccc}
  \mathcal{HL}^{CJ}(X) &\longrightarrow& \mathcal{HL}^{BGO}(X) &\longrightarrow&  \mathcal{HL}^{BGO}_k(X)&
  \longrightarrow& \mathcal{HL}^{CJ}_k(X)\\
  x &\longmapsto &|x| &\longmapsto &k|x| &\longmapsto& kx \,\text{ or } -kx.
\end{array}
\end{equation}

\begin{exa}
In Example \ref{veryfirstabacus} above, we have
$\mathcal{HL}_2^{BGO} (X)=\{2,4 \,|\,0,2\}=\mathcal{HL}_2^{CJ} (X)$. 
\end{exa}

\subsubsection{Charged hooks} 
Let $s \in \N$. If $X=(a_1,\ldots,a_m)$ is a $\beta$-set, then we set \index{$X[s]$}  $X[s]:=(0,1,\ldots,s-1,$ $a_1+s,\ldots,a_m+s)$ if $s \neq 0$, and $X[0]:=X$.
Let ${\bf s}=(s_1,\ldots,s_l) \in \mathbb{N}^l$. 
If $X=(X_1,\ldots,X_l)$ is an $l$-symbol, then we denote by \index{$X[\textbf{s}]$}. $X[\textbf{s}]$ the $l$-symbol $(X_1[s_1],\ldots,X_l[s_l])$. 
Obviously, $(a,b,i,j)$ is a hook of $X$ if and only if $(a+s_i,b+s_j,i,j)$ is a hook of $X[\textbf{s}]$. 
This in fact characterizes all the hooks of $X[\textbf{s}]$. Moreover, we obviously have
that $(a,b,i,j)$ is a CJ-hook of $X$ if and only if $(a+s_i,b+s_j,i,j)$ is a CJ-hook of $X[\textbf{s}]$, because the translation does not affect the number of empty spots in the abacus. This is of course not the case for BGO-hooks:

\begin{exa}\label{firstabacus} 
  Let $l=2$ and $X=((0,2,3),(0,1,2))$.
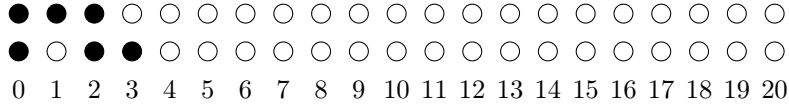
\begin{center}
\begin{tikzpicture}[scale=0.5, bb/.style={draw,circle,fill,minimum size=2.5mm,inner sep=0pt,outer sep=0pt}, wb/.style={draw,circle,fill=white,minimum size=2.5mm,inner sep=0pt,outer sep=0pt}]
	
	\node [] at (11,-1) {20};
	\node [] at (10,-1) {19};
	\node [] at (9,-1) {18};
	\node [] at (8,-1) {17};
	\node [] at (7,-1) {16};
	\node [] at (6,-1) {15};
	\node [] at (5,-1) {14};
	\node [] at (4,-1) {13};
	\node [] at (3,-1) {12};
	\node [] at (2,-1) {11};
	\node [] at (1,-1) {10};
	\node [] at (0,-1) {9};
	\node [] at (-1,-1) {8};
	\node [] at (-2,-1) {7};
	\node [] at (-3,-1) {6};
	\node [] at (-4,-1) {5};
	\node [] at (-5,-1) {4};
	\node [] at (-6,-1) {3};
	\node [] at (-7,-1) {2};
	\node [] at (-8,-1) {1};
	\node [] at (-9,-1) {0};

	\node [wb] at (11,0) {};
	\node [wb] at (10,0) {};
	\node [wb] at (9,0) {};
	\node [wb] at (8,0) {};
	\node [wb] at (7,0) {};
	\node [wb] at (6,0) {};
	\node [wb] at (5,0) {};
	\node [wb] at (4,0) {};
	\node [wb] at (3,0) {};
	\node [wb] at (2,0) {};
	\node [wb] at (1,0) {};
	\node [wb] at (0,0) {};
	\node [wb] at (-1,0) {};
	\node [wb] at (-2,0) {};
	\node [wb] at (-3,0) {};
	\node [wb] at (-4,0) {};
	\node [wb] at (-5,0) {};
	\node [bb] at (-6,0) {};
	\node [bb] at (-7,0) {};
	\node [wb] at (-8,0) {};
	\node [bb] at (-9,0) {};
	
	\node [wb] at (11,1) {};
	\node [wb] at (10,1) {};
	\node [wb] at (9,1) {};
	\node [wb] at (8,1) {};
	\node [wb] at (7,1) {};
	\node [wb] at (6,1) {};
	\node [wb] at (5,1) {};
	\node [wb] at (4,1) {};
	\node [wb] at (3,1) {};
	\node [wb] at (2,1) {};
	\node [wb] at (1,1) {};
	\node [wb] at (0,1) {};
	\node [wb] at (-1,1) {};
	\node [wb] at (-2,1) {};
	\node [wb] at (-3,1) {};
	\node [wb] at (-4,1) {};
	\node [wb] at (-5,1) {};
	\node [wb] at (-6,1) {};
	\node [bb] at (-7,1) {};
	\node [bb] at (-8,1) {};
	\node [bb] at (-9,1) {};
	
	\end{tikzpicture}
\end{center}
Then $\mathcal{HL}^{CJ} (X)=\{1,2 \,|\,0,-1\}$ and $\mathcal{HL}^{BGO} (X)=\{1,2 \,|\,0,1\}$. 
Now, for ${\bf s}=(1,4)$, we have $X[\textbf{s}]=((0,1,3,4),(0,1,2,3,4,5,6))$.
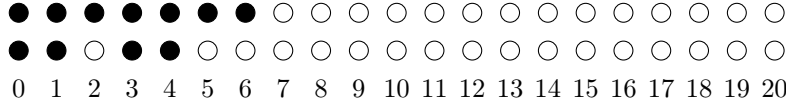
\begin{center}
\begin{tikzpicture}[scale=0.5, bb/.style={draw,circle,fill,minimum size=2.5mm,inner sep=0pt,outer sep=0pt}, wb/.style={draw,circle,fill=white,minimum size=2.5mm,inner sep=0pt,outer sep=0pt}]
	
	\node [] at (11,-1) {20};
	\node [] at (10,-1) {19};
	\node [] at (9,-1) {18};
	\node [] at (8,-1) {17};
	\node [] at (7,-1) {16};
	\node [] at (6,-1) {15};
	\node [] at (5,-1) {14};
	\node [] at (4,-1) {13};
	\node [] at (3,-1) {12};
	\node [] at (2,-1) {11};
	\node [] at (1,-1) {10};
	\node [] at (0,-1) {9};
	\node [] at (-1,-1) {8};
	\node [] at (-2,-1) {7};
	\node [] at (-3,-1) {6};
	\node [] at (-4,-1) {5};
	\node [] at (-5,-1) {4};
	\node [] at (-6,-1) {3};
	\node [] at (-7,-1) {2};
	\node [] at (-8,-1) {1};
	\node [] at (-9,-1) {0};

	\node [wb] at (11,0) {};
	\node [wb] at (10,0) {};
	\node [wb] at (9,0) {};
	\node [wb] at (8,0) {};
	\node [wb] at (7,0) {};
	\node [wb] at (6,0) {};
	\node [wb] at (5,0) {};
	\node [wb] at (4,0) {};
	\node [wb] at (3,0) {};
	\node [wb] at (2,0) {};
	\node [wb] at (1,0) {};
	\node [wb] at (0,0) {};
	\node [wb] at (-1,0) {};
	\node [wb] at (-2,0) {};
	\node [wb] at (-3,0) {};
	\node [wb] at (-4,0) {};
	\node [bb] at (-5,0) {};
	\node [bb] at (-6,0) {};
	\node [wb] at (-7,0) {};
	\node [bb] at (-8,0) {};
	\node [bb] at (-9,0) {};
	
	\node [wb] at (11,1) {};
	\node [wb] at (10,1) {};
	\node [wb] at (9,1) {};
	\node [wb] at (8,1) {};
	\node [wb] at (7,1) {};
	\node [wb] at (6,1) {};
	\node [wb] at (5,1) {};
	\node [wb] at (4,1) {};
	\node [wb] at (3,1) {};
	\node [wb] at (2,1) {};
	\node [wb] at (1,1) {};
	\node [wb] at (0,1) {};
	\node [wb] at (-1,1) {};
	\node [wb] at (-2,1) {};
	\node [bb] at (-3,1) {};
	\node [bb] at (-4,1) {};
	\node [bb] at (-5,1) {};
	\node [bb] at (-6,1) {};
	\node [bb] at (-7,1) {};
	\node [bb] at (-8,1) {};
	\node [bb] at (-9,1) {};
	
	\end{tikzpicture}
\end{center}
We have $\mathcal{HL}^{CJ} (X[\textbf{s}])=\{1,2 \,|\,-3,-4\}$ and $\mathcal{HL}^{BGO} (X[\textbf{s}])=\{1,2 \,|\,0,0,0,1,3,4\}$. 
  \end{exa}

Therefore, there exists a bijection
\begin{equation}\label{CJbij}
 \mathcal{H}^{CJ}(X) \longrightarrow  \mathcal{H}^{CJ}(X[\textbf{s}]), (a,b,i,j) \longmapsto (a+s_i,b+s_j,i,j),
 \end{equation}
 which, at the level of hook lengths, yields a bijection
 \begin{equation}\label{CJhlbij}
 \mathcal{HL}^{CJ}(X) \longrightarrow  \mathcal{HL}^{CJ}(X[\textbf{s}]), \mathfrak{hl}(a,b,i,j) \longmapsto \mathfrak{hl}(a,b,i,j)+s_i-s_j.
 \end{equation}

 \subsubsection{Charged scaled hooks}
 Let $k \in \mathbb{Z}_{>0}$ and ${\bf s} \in \mathbb{N}^l$.
 Let $X=(X_1,\ldots,X_l)$ be an $l$-symbol  with multicharge $(m_1,\ldots,m_l)$ and $m_1=\ldots=m_l$.
 Set 
 $$\mathcal{H}^{*}_{k,\textbf{s}}(X):=\{(a,b,i,j) \in \mathcal{H}^{*}((kX)[{\bf s}]) \,|\, (a-s_i)-(b-s_j) \equiv 0 \,{\rm mod}\,k\},$$
 where  $* \in \{BGO,CJ\}$. We have
 $\mathcal{H}^{*}_{1,\textbf{0}}(X)=\mathcal{H}^{*}(X)$, where $\textbf{0}=(0,\ldots,0) \in \N^l$.
 Combining \eqref{midbij}, \eqref{notmidbij2}, \eqref{notmidbij} and \eqref{CJbij}, we deduce that we have a bijection
 \begin{equation}
  \mathcal{H}^{CJ}_{k,\textbf{s}}(X) \longleftrightarrow \mathcal{H}^{BGO}_{1,\textbf{0}}(X).
  \end{equation}
In particular, at the level of hook lengths, this induces a bijection
\begin{equation}\label{CJvsBGO}
\mathcal{HL}^{CJ}_{k,\textbf{s}}(X) \longrightarrow \mathcal{HL}^{BGO}_{1,\textbf{0}}(X)[k;{\bf s}], \quad x \longmapsto \pm x,
\end{equation}
where $\mathcal{HL}^{*}_{1,\textbf{0}}(X)[k;{\bf s}] := \{k(a-b)+s_i-s_j\,|\,(a,b,i,j) \in \mathcal{H}^{*}_{1,\textbf{0}}(X)\}$.

\begin{exa}\label{firstabacus} 
  Let $l=2$ and $X=((0,2,3),(0,1,2))$ as before. We have $\mathcal{HL}^{BGO}_{1,\textbf{0}}(X)=\{1,2 \,|\,0,1\}$.
  For $k=2$ and ${\bf s}=(1,4)$, we have $\mathcal{HL}^{BGO}_{1,\textbf{0}}(X)[k;{\bf s}]=
\{2,4\,|\,3,5\}$ and $\mathcal{HL}^{CJ}_{1,\textbf{0}}(X)[k;{\bf s}]=
\{2,4\,|\,-3,-5\}$.
Moreover,  we have $(kX)[{\bf s}]=((0,1,5,7),(0,1,2,3,4,6,8))$.
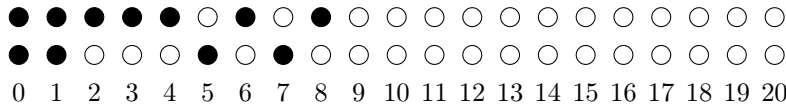
\begin{center}
\begin{tikzpicture}[scale=0.5, bb/.style={draw,circle,fill,minimum size=2.5mm,inner sep=0pt,outer sep=0pt}, wb/.style={draw,circle,fill=white,minimum size=2.5mm,inner sep=0pt,outer sep=0pt}]
	
	\node [] at (11,-1) {20};
	\node [] at (10,-1) {19};
	\node [] at (9,-1) {18};
	\node [] at (8,-1) {17};
	\node [] at (7,-1) {16};
	\node [] at (6,-1) {15};
	\node [] at (5,-1) {14};
	\node [] at (4,-1) {13};
	\node [] at (3,-1) {12};
	\node [] at (2,-1) {11};
	\node [] at (1,-1) {10};
	\node [] at (0,-1) {9};
	\node [] at (-1,-1) {8};
	\node [] at (-2,-1) {7};
	\node [] at (-3,-1) {6};
	\node [] at (-4,-1) {5};
	\node [] at (-5,-1) {4};
	\node [] at (-6,-1) {3};
	\node [] at (-7,-1) {2};
	\node [] at (-8,-1) {1};
	\node [] at (-9,-1) {0};

	\node [wb] at (11,0) {};
	\node [wb] at (10,0) {};
	\node [wb] at (9,0) {};
	\node [wb] at (8,0) {};
	\node [wb] at (7,0) {};
	\node [wb] at (6,0) {};
	\node [wb] at (5,0) {};
	\node [wb] at (4,0) {};
	\node [wb] at (3,0) {};
	\node [wb] at (2,0) {};
	\node [wb] at (1,0) {};
	\node [wb] at (0,0) {};
	\node [wb] at (-1,0) {};
	\node [bb] at (-2,0) {};
	\node [wb] at (-3,0) {};
	\node [bb] at (-4,0) {};
	\node [wb] at (-5,0) {};
	\node [wb] at (-6,0) {};
	\node [wb] at (-7,0) {};
	\node [bb] at (-8,0) {};
	\node [bb] at (-9,0) {};
	
	\node [wb] at (11,1) {};
	\node [wb] at (10,1) {};
	\node [wb] at (9,1) {};
	\node [wb] at (8,1) {};
	\node [wb] at (7,1) {};
	\node [wb] at (6,1) {};
	\node [wb] at (5,1) {};
	\node [wb] at (4,1) {};
	\node [wb] at (3,1) {};
	\node [wb] at (2,1) {};
	\node [wb] at (1,1) {};
	\node [wb] at (0,1) {};
	\node [bb] at (-1,1) {};
	\node [wb] at (-2,1) {};
	\node [bb] at (-3,1) {};
	\node [wb] at (-4,1) {};
	\node [bb] at (-5,1) {};
	\node [bb] at (-6,1) {};
	\node [bb] at (-7,1) {};
	\node [bb] at (-8,1) {};
	\node [bb] at (-9,1) {};

	\end{tikzpicture}
\end{center}
Then $\mathcal{H}^{CJ}_{k,\textbf{s}}(X)=\{(5,3,1,1),(7,3,1,1),(7,10,1,2), (8,3,2,1)\}$,
and so
$ \mathcal{HL}^{CJ}_{k,\textbf{s}}(X) =\{2,4 \,|\,-3,5\}$.
  \end{exa}

\subsection{Multipartitions}

A {\it partition} of (rank) $n$ is  a sequence of integers $(\lambda_1,\ldots,\lambda_r)$ such that $\lambda_1\geq \ldots \geq \lambda_r \geq 0$ and such that 
$\sum_{1\leq i\leq r} \lambda_i =n$. We consider that adding $0$'s to the partition does not change the partition.
  To each  $\beta$-set $X=(a_1,\ldots,a_m)$ we can canonically associate 
  a partition \index{$\Lambda (X$) (where $X$ is a $\beta$-set} $\Lambda (X)=(\lambda_1,\ldots,\lambda_m)$ such that, for all $i=1,\ldots,m$, we have $\lambda_i=a_{m-i+1}+i-m$. In the abacus configuration of $X$, $\lambda_i$ equals the number of empty spots on the left of $a_{m-i+1}$. 
  
        \begin{exa} If  $X=(0,3,4,6,8)$,  then we have $\Lambda (X)=(4,3,2,2)$.
        \end{exa}
  
  Conversely, to any partition $\lambda=(\lambda_1,\ldots,\lambda_r)$, we can associate a $\beta$-set $X(\lambda)=(a_1,\ldots,a_m)$,
  where $m:=\operatorname{min}\{i>0 \ |\ \lambda_i=0\}$ and such that, for all $i=1,\ldots,m$, we have $a_{m-i+1}=\lambda_i-i+m$.  Now, for any $s \in \N$, we have that $\Lambda(X(\lambda)[s])=\lambda$.

Now, let $l \in  \mathbb{Z}_{>0}$. An {\it $l$-partition} (or multipartition) of $n$  is an $l$-tuple of partitions $(\lambda^1,\ldots, \lambda^l)$ such that the sum of the ranks of the 
 $\lambda^j$'s is $n$. To each $l$-symbol $X=(X_1,\ldots,X_l)$, we can associate a multipartition \index{$\uLambda (X)$ (where $X$ is an $l$-symbol)}  $\uLambda (X)=(\Lambda (X_1),\ldots,\Lambda (X_l))$ together with a multicharge ${\bf s} (X)\in \mathbb{N}^l$ which is the multicharge of the symbol.
         \begin{exa} Let  $X=((0,1,2,5),(0,3,4))$ then we have  $\uLambda (X)=((2),(2,2))$ and ${\bf s} (X)=(4,3)$. 
        \end{exa}
 
Conversely, to any $l$-partition $\ulambda=(\lambda^1,\ldots, \lambda^l)$, we can associate an $l$-symbol
$X(\ulambda)=(X(\lambda^1),\ldots, X(\lambda^l))$. Everything we proved in \S\ref{Symbols} about hooks and hook lengths applies to $X(\ulambda)$. As before, for each ${\bf s} \in \N^l$, we have $\uLambda(X(\ulambda)[{\bf s}])=\ulambda$.  For the empty $l$-partition $\boldsymbol{\emptyset}$, we will take $X(\boldsymbol{\emptyset})$ to be empty.

 If now $(m_1,\ldots,m_l)$ is the multicharge of $X(\ulambda)$, set $m:={\rm max}(m_i)$ and
${\bf m}:=(m-m_1,\ldots,m-m_l)$. The $l$-symbol $X(\ulambda)[{\bf m}]$ represents the $l$-partition $\ulambda$ and has multicharge $(m,m,\ldots,m)$. It will be denoted by \index{$X^{1,\boldsymbol{0}}(\ulambda)$ (where $\ulambda$ is an $l$-partition)}$X^{1,\boldsymbol{0}}(\ulambda)$. We thus have
$X^{1,\boldsymbol{0}}(\ulambda)_j = (0,1,\ldots,m-m_j-1,\lambda^j_{m_j}-m_j+m,\ldots,\lambda^j_2-2+m,\lambda^j_1-1+m)$ for all $j=1,\ldots,l$. Equivalently, $X^{1,\boldsymbol{0}}(\ulambda)_j =(a_1^j,a_2^j,\ldots,a_m^j)$ where $a_{m-i+1}^j=\lambda_i^j-i+m$. Generalising this notation, we will denote by \index{$X^{k,{\bf s}}(\ulambda)$ (where $\ulambda$ is an $l$-partition)} $X^{k,{\bf s}}(\ulambda)$ the $l$-symbol $(kX^{1,\boldsymbol{0}}(\ulambda))[{\bf s}]$ for any $k \in \mathbb{Z}_{>0}$ and ${\bf s} \in \mathbb{N}^l$.

 \subsection{Cores and quotients}\label{core}
Fix $e\in \mathbb{Z}_{>0}$. Let $n\in \mathbb{N}$ and let $\lambda$ be a partition of $n$.  Let $X(\lambda)$ denote the corresponding set of $\beta$-numbers. We can now associate to $\lambda$ an $e$-symbol
 $Y(\lambda)=(Y_1,\ldots,Y_e)$
  where, for all $j\in \{1,\ldots,e\}$, $Y_j$ is the  set of increasing integers obtained by ordering the set
  $$\{ k\in \mathbb{N}\ |\ j-1+ke \in X(\lambda)\}.$$
  The abacus configuration of $Y(\lambda)$ is the \emph{$e$-abacus presentation} of $\lambda$.

Let $\ulambda$ be the $e$-partition such that $X (\ulambda)=Y(\lambda)$. Then $\ulambda$ is the {\it $e$-quotient} of $\lambda$. If now ${\bf s}=(s_1,\ldots,s_e)$ is the multicharge of $Y(\lambda)$, then the \emph{$e$-core} of $\lambda$ is the partition $\lambda^{\circ}$ such that $Y(\lambda^{\circ})=X(\boldsymbol{\emptyset})[{\bf s}]$, where $\boldsymbol{\emptyset}$ denotes the empty $e$-partition (this amounts to pushing all beads of $Y(\lambda)$ to the left).

  \begin{exa}
  Let $\lambda=(3,2,1,1,1)$ and $e=3$. Then  $X(\lambda)=(0,2,3,4,6,8)$ and thus
  $$Y(\lambda)=((0,1,2),(1),(0,2)).$$
  Thus, the $3$-quotient of $\lambda$ is $(\emptyset, (1),(1))$. The multicharge of $Y(\lambda)$ is ${\bf s}=(3,1,2)$  and 
  so 
  the $3$-core is the partition $\lambda^{\circ}$ such that 
  $$Y(\lambda^{\circ})=((0,1,2),(1),(0,1)),$$
  that is, $\lambda^{\circ}=(1,1)$.
  \end{exa}

  Now, the multiset $\mathcal{HL} (\lambda)$ of hook lengths of a partition $\lambda$ is the multiset $\mathcal{HL}^{BGO} (X (\lambda))$.
By \cite[Theorem 4.4]{BGO}, we have:
\begin{Prop} Let $\tilde{\bf s}=(es_1,es_2+1,\ldots,es_e+(e-1))$. As multisets, we have
$$\mathcal{HL} ( \lambda)=\mathcal{HL} (\lambda^{\circ})\cup |\mathcal{HL}^{BGO}_{1,{\bf 0}} (X^{1,\boldsymbol{0}}(\ulambda))[e;\tilde{\bf s}]|,$$
where $|\cdot|$ means that we  take the absolute value of each element in the multiset.
\end{Prop}

Using \eqref{CJvsBGO}, we deduce:
 \begin{Cor}\label{particore}
  Let $\tilde{\bf s}=(es_1,es_2+1,\ldots,es_e+(e-1))$. As multisets, we have
$$\mathcal{HL} ( \lambda)=\mathcal{HL} (\lambda^{\circ})\cup 
|\mathcal{HL}^{CJ}_{e,\tilde{\bf s}} (X^{1,\boldsymbol{0}}(\ulambda)) |,$$
where $|\cdot|$ means that we  take the absolute value of each element in the multiset.
\end{Cor}
 A consequence of the above result  in terms of Schur elements will be given in the next section.

   \begin{exa}\label{coreexample}
   In the previous example of $\lambda=(3,2,1,1,1)$, we have $\mathcal{HL} ( \lambda)=\{1,1,1,2,3,3,5,7\}$ and
   $\mathcal{HL} (\lambda^{\circ})=\{1,2\}.$ 
   Moreover, ${\bf m}=(0,2,1)$ and $\tilde{\bf s}=(9,4,8)$.
   The abacus configuration for $X^{1,{\bf 0}}(\ulambda)$ is:
   \begin{center}
   \begin{tikzpicture}[scale=0.5, bb/.style={draw,circle,fill,minimum size=2.5mm,inner sep=0pt,outer sep=0pt}, wb/.style={draw,circle,fill=white,minimum size=2.5mm,inner sep=0pt,outer sep=0pt}]
	
	\node [] at (11,-1) {20};
	\node [] at (10,-1) {19};
	\node [] at (9,-1) {18};
	\node [] at (8,-1) {17};
	\node [] at (7,-1) {16};
	\node [] at (6,-1) {15};
	\node [] at (5,-1) {14};
	\node [] at (4,-1) {13};
	\node [] at (3,-1) {12};
	\node [] at (2,-1) {11};
	\node [] at (1,-1) {10};
	\node [] at (0,-1) {9};
	\node [] at (-1,-1) {8};
	\node [] at (-2,-1) {7};
	\node [] at (-3,-1) {6};
	\node [] at (-4,-1) {5};
	\node [] at (-5,-1) {4};
	\node [] at (-6,-1) {3};
	\node [] at (-7,-1) {2};
	\node [] at (-8,-1) {1};
	\node [] at (-9,-1) {0};

	\node [wb] at (11,0) {};
	\node [wb] at (10,0) {};
	\node [wb] at (9,0) {};
	\node [wb] at (8,0) {};
	\node [wb] at (7,0) {};
	\node [wb] at (6,0) {};
	\node [wb] at (5,0) {};
	\node [wb] at (4,0) {};
	\node [wb] at (3,0) {};
	\node [wb] at (2,0) {};
	\node [wb] at (1,0) {};
	\node [wb] at (0,0) {};
	\node [wb] at (-1,0) {};
	\node [wb] at (-2,0) {};
	\node [wb] at (-3,0) {};
	\node [wb] at (-4,0) {};
	\node [wb] at (-5,0) {};
	\node [wb] at (-6,0) {};
	\node [bb] at (-7,0) {};
	\node [bb] at (-8,0) {};
	\node [bb] at (-9,0) {};
	
	\node [wb] at (11,1) {};
	\node [wb] at (10,1) {};
	\node [wb] at (9,1) {};
	\node [wb] at (8,1) {};
	\node [wb] at (7,1) {};
	\node [wb] at (6,1) {};
	\node [wb] at (5,1) {};
	\node [wb] at (4,1) {};
	\node [wb] at (3,1) {};
	\node [wb] at (2,1) {};
	\node [wb] at (1,1) {};
	\node [wb] at (0,1) {};
	\node [wb] at (-1,1) {};
	\node [wb] at (-2,1) {};
	\node [wb] at (-3,1) {};
	\node [wb] at (-4,1) {};
	\node [wb] at (-5,1) {};
	\node [bb] at (-6,1) {};
	\node [wb] at (-7,1) {};
	\node [bb] at (-8,1) {};
	\node [bb] at (-9,1) {};
	
	\node [wb] at (11,2) {};
	\node [wb] at (10,2) {};
	\node [wb] at (9,2) {};
	\node [wb] at (8,2) {};
	\node [wb] at (7,2) {};
	\node [wb] at (6,2) {};
	\node [wb] at (5,2) {};
	\node [wb] at (4,2) {};
	\node [wb] at (3,2) {};
	\node [wb] at (2,2) {};
	\node [wb] at (1,2) {};
	\node [wb] at (0,2) {};
	\node [wb] at (-1,2) {};
	\node [wb] at (-2,2) {};
	\node [wb] at (-3,2) {};
	\node [wb] at (-4,2) {};
	\node [wb] at (-5,2) {};
	\node [bb] at (-6,2) {};
	\node [wb] at (-7,2) {};
	\node [bb] at (-8,2) {};
	\node [bb] at (-9,2) {};
	
	\end{tikzpicture}
\end{center}
 Then  $\mathcal{H}^{BGO}_{1,{\bf 0}} (\ulambda)=\{(3,2,2,2),(3,3,2,1),(3,2,2,3),(3,2,3,3),(3,3,3,1),(3,2,3,2)\}$, and so  $\mathcal{HL}^{BGO}_{1,{\bf 0}} (\ulambda)[e;\tilde{\bf s}]=\{3,-5,-1,3,-1,7\}$. On the other hand,
 $X^{3,\tilde{\bf s}} (\ulambda)$ has abacus configuration

  
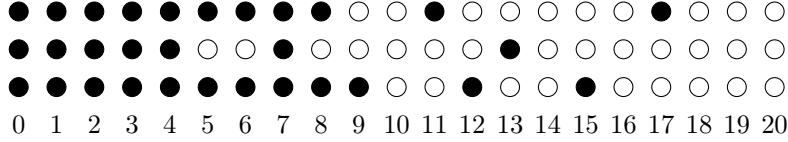
\begin{center}
   \begin{tikzpicture}[scale=0.5, bb/.style={draw,circle,fill,minimum size=2.5mm,inner sep=0pt,outer sep=0pt}, wb/.style={draw,circle,fill=white,minimum size=2.5mm,inner sep=0pt,outer sep=0pt}]
	
	\node [] at (11,-1) {20};
	\node [] at (10,-1) {19};
	\node [] at (9,-1) {18};
	\node [] at (8,-1) {17};
	\node [] at (7,-1) {16};
	\node [] at (6,-1) {15};
	\node [] at (5,-1) {14};
	\node [] at (4,-1) {13};
	\node [] at (3,-1) {12};
	\node [] at (2,-1) {11};
	\node [] at (1,-1) {10};
	\node [] at (0,-1) {9};
	\node [] at (-1,-1) {8};
	\node [] at (-2,-1) {7};
	\node [] at (-3,-1) {6};
	\node [] at (-4,-1) {5};
	\node [] at (-5,-1) {4};
	\node [] at (-6,-1) {3};
	\node [] at (-7,-1) {2};
	\node [] at (-8,-1) {1};
	\node [] at (-9,-1) {0};

	\node [wb] at (11,0) {};
	\node [wb] at (10,0) {};
	\node [wb] at (9,0) {};
	\node [wb] at (8,0) {};
	\node [wb] at (7,0) {};
	\node [bb] at (6,0) {};
	\node [wb] at (5,0) {};
	\node [wb] at (4,0) {};
	\node [bb] at (3,0) {};
	\node [wb] at (2,0) {};
	\node [wb] at (1,0) {};
	\node [bb] at (0,0) {};
	\node [bb] at (-1,0) {};
	\node [bb] at (-2,0) {};
	\node [bb] at (-3,0) {};
	\node [bb] at (-4,0) {};
	\node [bb] at (-5,0) {};
	\node [bb] at (-6,0) {};
	\node [bb] at (-7,0) {};
	\node [bb] at (-8,0) {};
	\node [bb] at (-9,0) {};
	
	\node [wb] at (11,1) {};
	\node [wb] at (10,1) {};
	\node [wb] at (9,1) {};
	\node [wb] at (8,1) {};
	\node [wb] at (7,1) {};
	\node [wb] at (6,1) {};
	\node [wb] at (5,1) {};
	\node [bb] at (4,1) {};
	\node [wb] at (3,1) {};
	\node [wb] at (2,1) {};
	\node [wb] at (1,1) {};
	\node [wb] at (0,1) {};
	\node [wb] at (-1,1) {};
	\node [bb] at (-2,1) {};
	\node [wb] at (-3,1) {};
	\node [wb] at (-4,1) {};
	\node [bb] at (-5,1) {};
	\node [bb] at (-6,1) {};
	\node [bb] at (-7,1) {};
	\node [bb] at (-8,1) {};
	\node [bb] at (-9,1) {};
	
	\node [wb] at (11,2) {};
	\node [wb] at (10,2) {};
	\node [wb] at (9,2) {};
	\node [bb] at (8,2) {};
	\node [wb] at (7,2) {};
	\node [wb] at (6,2) {};
	\node [wb] at (5,2) {};
	\node [wb] at (4,2) {};
	\node [wb] at (3,2) {};
	\node [bb] at (2,2) {};
	\node [wb] at (1,2) {};
	\node [wb] at (0,2) {};
	\node [bb] at (-1,2) {};
	\node [bb] at (-2,2) {};
	\node [bb] at (-3,2) {};
	\node [bb] at (-4,2) {};
	\node [bb] at (-5,2) {};
	\node [bb] at (-6,2) {};
	\node [bb] at (-7,2) {};
	\node [bb] at (-8,2) {};
	\node [bb] at (-9,2) {};
	
	\end{tikzpicture}
\end{center}
 Then $\mathcal{H}^{CJ}_{3,\tilde{\textbf{s}}}(X^{1,{\bf 0}}(\ulambda))=\{(13,10,2,2),(13,18,2,1),(13,14,2,3),(17,14,3,3),(17,18,3,1),(17,10,3,2)\}$,
and so
$ \mathcal{HL}^{CJ}_{3,\tilde{\textbf{s}}}(X^{1,{\bf 0}}(\ulambda)) =\{3,-5,-1,3,-1,7\}$.
   \end{exa}
 
Let now $X=(X_1,\ldots, X_l)$ be an $l$-symbol and let $t\in\mathbb{Z}_{>0}$. Let $d\in \{0,1,\ldots,l-1\}$.  Then 
  the \emph{$[d,t]$-core of $X$} is an $l$-symbol, denoted by \index{$X[d,t]$ (where $X$ is an $l$-symbol)} $X[d,t]$, obtained as follows. For each hook $(a,b,i,j)$ such that 
  $i-j \equiv d\, {\rm mod}\, l$ 
  and  $a-b=t$, 
   we replace $X^i$ by $X^i\setminus \{a\}$ and $X^j$ by $X^j\cup \{b\}$. We say that we ``remove'' the hook $(a,b,i,j)$. Continuing this process as long as necessary,
the resulting symbol has no hook $(a,b,i,j)$  such that $i-j \equiv d\, {\rm mod}\, l$ and  $a-b=t$, and it does not depend on the order in which the hooks are removed
(see \cite{BGO}).

\begin{exa}\label{firstabacus} 
  Let $l=2$ and $X=((0,2,4,7),(0,2,3,4,6,8))$. 
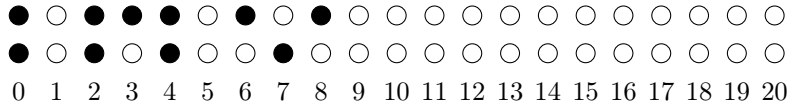
\begin{center}
\begin{tikzpicture}[scale=0.5, bb/.style={draw,circle,fill,minimum size=2.5mm,inner sep=0pt,outer sep=0pt}, wb/.style={draw,circle,fill=white,minimum size=2.5mm,inner sep=0pt,outer sep=0pt}]
	
	\node [] at (11,-1) {20};
	\node [] at (10,-1) {19};
	\node [] at (9,-1) {18};
	\node [] at (8,-1) {17};
	\node [] at (7,-1) {16};
	\node [] at (6,-1) {15};
	\node [] at (5,-1) {14};
	\node [] at (4,-1) {13};
	\node [] at (3,-1) {12};
	\node [] at (2,-1) {11};
	\node [] at (1,-1) {10};
	\node [] at (0,-1) {9};
	\node [] at (-1,-1) {8};
	\node [] at (-2,-1) {7};
	\node [] at (-3,-1) {6};
	\node [] at (-4,-1) {5};
	\node [] at (-5,-1) {4};
	\node [] at (-6,-1) {3};
	\node [] at (-7,-1) {2};
	\node [] at (-8,-1) {1};
	\node [] at (-9,-1) {0};

	\node [wb] at (11,0) {};
	\node [wb] at (10,0) {};
	\node [wb] at (9,0) {};
	\node [wb] at (8,0) {};
	\node [wb] at (7,0) {};
	\node [wb] at (6,0) {};
	\node [wb] at (5,0) {};
	\node [wb] at (4,0) {};
	\node [wb] at (3,0) {};
	\node [wb] at (2,0) {};
	\node [wb] at (1,0) {};
	\node [wb] at (0,0) {};
	\node [wb] at (-1,0) {};
	\node [bb] at (-2,0) {};
	\node [wb] at (-3,0) {};
	\node [wb] at (-4,0) {};
	\node [bb] at (-5,0) {};
	\node [wb] at (-6,0) {};
	\node [bb] at (-7,0) {};
	\node [wb] at (-8,0) {};
	\node [bb] at (-9,0) {};
	
	\node [wb] at (11,1) {};
	\node [wb] at (10,1) {};
	\node [wb] at (9,1) {};
	\node [wb] at (8,1) {};
	\node [wb] at (7,1) {};
	\node [wb] at (6,1) {};
	\node [wb] at (5,1) {};
	\node [wb] at (4,1) {};
	\node [wb] at (3,1) {};
	\node [wb] at (2,1) {};
	\node [wb] at (1,1) {};
	\node [wb] at (0,1) {};
	\node [bb] at (-1,1) {};
	\node [wb] at (-2,1) {};
	\node [bb] at (-3,1) {};
	\node [wb] at (-4,1) {};
	\node [bb] at (-5,1) {};
	\node [bb] at (-6,1) {};
	\node [bb] at (-7,1) {};
	\node [wb] at (-8,1) {};
	\node [bb] at (-9,1) {};

	\end{tikzpicture}
\end{center}
 Then $X[0,3]=((0,1,2,4),(0,1,2,3,5,6))$ and  $X[1,3]=((0,1,2,3,5),(0,1,2,3,4))$.
  \end{exa}

\begin{Rem}
For any $e \in \mathbb{Z}_{>0}$, it is easy to see that $X[0,e]=(\lambda(X_1)^{\circ},\ldots,\lambda(X_l)^{\circ})$.
\end{Rem}

On the other hand,  in \cite{JL}, for $e \in \mathbb{Z}_{>0}$ and ${\bf s}\in \N^l$, we have a notion of an \emph{$(e,{\bf s})$-core for a multipartition} $\ulambda$ which can be defined in terms of the symbol $X^{1,{\bf s}}(\ulambda)$. For a general $l$-symbol $X$, this  $e$-core  is obtained as follows:
 \begin{enumerate}
 \item If $(a,a,i,j)$ is a hook with $i<j$, then  $X^i$ is replaced by $X^i\setminus \{a\}$ and $X^j$ by $X^j\cup \{a\}$. \smallbreak
 \item If $(a,a-e,l,1)$ is a hook, then  $X^l$ is replaced by $X^l \setminus \{a\}$ and $X^1$ by $X^1\cup \{a-e\}$.
 \end{enumerate}
Continuing this process as long as necessary,
 the resulting symbol is the $e$-core of $X$ and is denoted by \index{$X^{\circ}$} $X^{\circ}$, while its multicharge is denoted by \index{${\bf s}^{\circ}$}${\bf s}^{\circ}$. 
 Note that the above process for calculating the $e$-core can be 
 replaced by the following one:
  \begin{enumerate}
  \item  Applying Step (1) of the Jacon--Lecouvey algorithm. \smallbreak
    \item Taking the $[l-1,e]$-core. \smallbreak
 \item Repeating this process as long as necessary.
  \end{enumerate}
Moreover, note that 
\begin{itemize}
\item Step (1) does not affect the non-zero BGO-hook lengths. \smallbreak
 \item $X^\circ$ has only non-zero CJ-hook lengths, that is, $0 \notin \mathcal{HL}^{CJ} (X^{\circ})$.
 \end{itemize}
 If $X=X^{1,{\bf s}}(\ulambda)$, then the $l$-partition $\ulambda(X^\circ)$ is the $(e,{\bf s})$-core of $\ulambda$ and is denoted by $\ulambda^{\circ}$.
 
 \medskip
 
We now compare our two  notions of hook lengths   in the general case. Let $\ulambda$ be an $l$-partition. First, we need a useful result (\cite[Theorem 3.3]{BGO}): for $k \in \N$ and ${\bf s},{\bf x} \in \mathbb{N}^l$, as multisets we have
$$\mathcal{HL}^{BGO}(X^{1,{\bf s}} (\ulambda))[k;{\bf x}] =\mathcal{HL}^{BGO}( X^{1,\boldsymbol{0}}(\ulambda))[k;{\bf x}'] \cup\mathcal{HL}^{BGO} (X^{1,{\bf s}}(\boldsymbol{\emptyset}))[k;{\bf x}]$$
where ${\bf x}'=(x_1+s_1k,\ldots,x_l+s_l k)$ and we take $X^{1,{\bf s}}(\boldsymbol{\emptyset})$ to have the same multicharge
  as $X^{1,{\bf s}}(\ulambda)$.
If we now take ${\bf x}=\boldsymbol{0}$ and $k=1$, we get
$$\mathcal{HL}^{BGO}(X^{1,{\bf s}} (\ulambda))=\mathcal{HL}^{BGO}( X^{1,\boldsymbol{0}}(\ulambda))[1;{\bf s}] \cup\mathcal{HL}^{BGO} (X^{1,{\bf s}}(\boldsymbol{\emptyset})),$$
which in turn yields the following.
\begin{Prop}\label{propositionabove} 
Let $\ulambda$ be an $l$-partition and ${\bf s} \in \mathbb{N}^l$. As multisets, we have
$$\mathcal{HL}^{BGO}(X^{1,{\bf s}} (\ulambda))=|\mathcal{HL}^{CJ}( X^{1,{\bf s}}(\ulambda))|\cup\mathcal{HL}^{BGO} (X^{1,{\bf s}}(\boldsymbol{\emptyset})),$$
where $|\cdot|$ means that we  take the absolute value of each element in the multiset.
\end{Prop}
    The following proposition is a reformulation of \cite[Theorem 5.4]{BGO}.
    \begin{Th} Let $t\in \mathbb{Z}_{>0}$and  $d\in \{0,1,\ldots,l-1\}$.
    The non-zero elements of $\mathcal{HL}^{BGO} (X[d,t])$  are contained in the non-zero elements of 
      $\mathcal{HL}^{BGO} (X)$.
    \end{Th}

    From this theorem  and the above discussion, we obtain:
 
        \begin{Prop}\label{prop2.3} Let $\ulambda$ be an $l$-partition and ${\bf s} \in \mathbb{N}^l$. Set  $X:=X^{1,{\bf s}}(\ulambda)$. Then 
 the non-zero elements of  $|\mathcal{HL}^{CJ} (X^{\circ})|\cup  \mathcal{HL}^{BGO} (X^{1,{\bf s}^{\circ}}(\boldsymbol{\emptyset})) $  are contained in 
      $|\mathcal{HL}^{CJ} (X)|\cup  \mathcal{HL}^{BGO} (X^{1,{\bf s}}(\boldsymbol{\emptyset}))$. In particular, for ${\bf s}=\boldsymbol{0}$,  the elements of  $|\mathcal{HL}^{CJ} (X^{\circ})| $  are contained in 
      $|\mathcal{HL}^{CJ} (X)|$.
    \end{Prop}

    \subsection{The $a$-function}\label{afonc}

For any $l$-symbol $X$, if we write the elements of $X$ as $(\kappa_1,\ldots,\kappa_m)$ in decreasing order, then we define \index{$\mathfrak{a}(X)$}
$$\mathfrak{a}(X):=\sum_{1\leq i\leq m} (i-1) \kappa_i.$$

 Let $\ulambda=(\lambda^1,\ldots,\lambda^{l})$ be an $l$-partition. Let ${\bf s}=(s_1,\ldots,s_l) \in \mathbb{N}^l$ and let $k \in \mathbb{N}$. Recall that  if $X^{1,{\bf 0}}(\ulambda)$ has multicharge $(m_{\ulambda},m_{\ulambda},\ldots,m_{\ulambda})$, then
 $X^{k,{\bf s}}(\ulambda)$ has multicharge $(m_{\ulambda}+s_1,\ldots,m_{\ulambda}+s_l)$.

 The $a$-value of $\ulambda$ is  \cite[Proposition 5.5.11]{GJ}:\index{$a_{{\bf s},k} (\ulambda)$}
   $$a_{{\bf s},k} (\ulambda)=\mathfrak{a}(X^{k,{\bf s}}(\ulambda))-
   \mathfrak{a}(X^{k,{\bf s}}(\boldsymbol{\emptyset}))$$
  where we take $X^{k,{\bf s}}(\boldsymbol{\emptyset})$ to have the same multicharge
  as $X^{k,{\bf s}}(\ulambda)$.

We will simply write \index{$a(\ulambda)$ } $a(\ulambda)$ for $a_{{\bf 0},1} (\ulambda)$. Note that for $l=1$, the charge does not affect the value of the $a$-function. Moreover, we have $a_{{\bf s},k} (\ulambda)=ka_{{\bf s},1} (\ulambda)$.

We will now use the notation in the beginning of the last subsection.
 Fix $e\in \Z_{>0}$. Let $n\in \mathbb{N}$ and let $\lambda$ be a partition of $n$. Let $\ulambda$ be the $e$-quotient of $\lambda$ and let $\lambda^{\circ}$ be the $e$-core of $\lambda$.
 Let ${\bf s}=(s_1,\ldots,s_e)$ be the multicharge of $X(\ulambda)$  and 
$\tilde{\bf s}:=(es_1,es_2+1,\ldots,es_e+(e-1))$.
 
 \begin{Prop}\label{a-function} We have
\begin{equation}\label{a-funeq}
a(\lambda)=a (\lambda^{\circ})+a_{\tilde{\bf s},e} (\ulambda).
\end{equation}
\end{Prop}

 \begin{proof}
First note that, for $j=1,\ldots,e$, we have 
$x \in Y(\lambda)_j$ if and only if $e x +j-1+em_{\ulambda} \in X^{e,\tilde{\bf s}} (\ulambda)_j$. 

On the other hand, by definition of $Y(\lambda)$, we have $x \in Y(\lambda)_j$ if and only if $e x +j-1 \in X^{1,{\bf 0}} (\lambda)$. 
Note that the function $x\mapsto e x +j-1$ is a strictly increasing function.

Similarly, we have  
$y \in Y(\lambda^\circ)_j$ if and only if $e y +j-1+em_{\ulambda} \in X^{e,\tilde{\bf s}} (\boldsymbol{\emptyset})_j$. 

On the other hand, by definition of $Y(\lambda^\circ)$, we have $y \in Y(\lambda^\circ)_j$ if and only if $e y +j-1 \in X^{1,{\bf t}}  (\lambda^\circ)$, where ${\bf t}$ is taken so that $\#X^{1,{\bf t}}  (\lambda^\circ)=\#Y(\lambda^\circ)$.

Therefore, we obtain
$$a(\lambda)-a(\lambda^\circ)=\mathfrak{a}(X(\lambda))-\mathfrak{a}(X(\lambda^\circ))=\mathfrak{a}(X^{e,\tilde{\bf s}} (\ulambda))-\mathfrak{a}(X^{e,\tilde{\bf s}} (\boldsymbol{\emptyset}))
$$
because the terms $em_\ulambda$ cancel out (since their number is equal to $\#Y(\lambda)=\#Y(\lambda^\circ)$), and the same holds for all remaining (smallest) elements of $ X^{e,\tilde{\bf s}} (\ulambda)$ and $X^{e,\tilde{\bf s}} (\boldsymbol{\emptyset})$.
 \end{proof}

    \begin{exa}  Let us revisit  Example \ref{coreexample}.
  Let $\lambda=(3,2,1,1,1)$ and $e=3$. Then $\lambda^{\circ}=(1,1)$ and $\ulambda=(\emptyset,(1),(1))$.
  Moreover, $\tilde{\bf s}=(9,4,8)$.
  The abacus configuration of $X^{1,{\bf 0}} (\lambda)$ is
    \begin{center}
   \begin{tikzpicture}[scale=0.5, bb/.style={draw,circle,fill,minimum size=2.5mm,inner sep=0pt,outer sep=0pt}, wb/.style={draw,circle,fill=white,minimum size=2.5mm,inner sep=0pt,outer sep=0pt}]
	
	\node [] at (11,-1) {20};
	\node [] at (10,-1) {19};
	\node [] at (9,-1) {18};
	\node [] at (8,-1) {17};
	\node [] at (7,-1) {16};
	\node [] at (6,-1) {15};
	\node [] at (5,-1) {14};
	\node [] at (4,-1) {13};
	\node [] at (3,-1) {12};
	\node [] at (2,-1) {11};
	\node [] at (1,-1) {10};
	\node [] at (0,-1) {9};
	\node [] at (-1,-1) {8};
	\node [] at (-2,-1) {7};
	\node [] at (-3,-1) {6};
	\node [] at (-4,-1) {5};
	\node [] at (-5,-1) {4};
	\node [] at (-6,-1) {3};
	\node [] at (-7,-1) {2};
	\node [] at (-8,-1) {1};
	\node [] at (-9,-1) {0};

	\node [wb] at (11,0) {};
	\node [wb] at (10,0) {};
	\node [wb] at (9,0) {};
	\node [wb] at (8,0) {};
	\node [wb] at (7,0) {};
	\node [wb] at (6,0) {};
	\node [wb] at (5,0) {};
	\node [wb] at (4,0) {};
	\node [wb] at (3,0) {};
	\node [wb] at (2,0) {};
	\node [wb] at (1,0) {};
	\node [wb] at (0,0) {};
	\node [bb] at (-1,0) {};
	\node [wb] at (-2,0) {};
	\node [bb] at (-3,0) {};
	\node [wb] at (-4,0) {};
	\node [bb] at (-5,0) {};
	\node [bb] at (-6,0) {};
	\node [bb] at (-7,0) {};
	\node [wb] at (-8,0) {};
	\node [bb] at (-9,0) {};
	\end{tikzpicture}
	\end{center}
The abacus configuration of $Y(\lambda)$ is		
 \begin{center}
   \begin{tikzpicture}[scale=0.5, bb/.style={draw,circle,fill,minimum size=2.5mm,inner sep=0pt,outer sep=0pt}, wb/.style={draw,circle,fill=white,minimum size=2.5mm,inner sep=0pt,outer sep=0pt}]
	
	\node [] at (11,-1) {20};
	\node [] at (10,-1) {19};
	\node [] at (9,-1) {18};
	\node [] at (8,-1) {17};
	\node [] at (7,-1) {16};
	\node [] at (6,-1) {15};
	\node [] at (5,-1) {14};
	\node [] at (4,-1) {13};
	\node [] at (3,-1) {12};
	\node [] at (2,-1) {11};
	\node [] at (1,-1) {10};
	\node [] at (0,-1) {9};
	\node [] at (-1,-1) {8};
	\node [] at (-2,-1) {7};
	\node [] at (-3,-1) {6};
	\node [] at (-4,-1) {5};
	\node [] at (-5,-1) {4};
	\node [] at (-6,-1) {3};
	\node [] at (-7,-1) {2};
	\node [] at (-8,-1) {1};
	\node [] at (-9,-1) {0};

	\node [wb] at (11,0) {};
	\node [wb] at (10,0) {};
	\node [wb] at (9,0) {};
	\node [wb] at (8,0) {};
	\node [wb] at (7,0) {};
	\node [wb] at (6,0) {};
	\node [wb] at (5,0) {};
	\node [wb] at (4,0) {};
	\node [wb] at (3,0) {};
	\node [wb] at (2,0) {};
	\node [wb] at (1,0) {};
	\node [wb] at (0,0) {};
	\node [wb] at (-1,0) {};
	\node [wb] at (-2,0) {};
	\node [wb] at (-3,0) {};
	\node [wb] at (-4,0) {};
	\node [wb] at (-5,0) {};
	\node [wb] at (-6,0) {};
	\node [bb] at (-7,0) {};
	\node [bb] at (-8,0) {};
	\node [bb] at (-9,0) {};
	
	\node [wb] at (11,1) {};
	\node [wb] at (10,1) {};
	\node [wb] at (9,1) {};
	\node [wb] at (8,1) {};
	\node [wb] at (7,1) {};
	\node [wb] at (6,1) {};
	\node [wb] at (5,1) {};
	\node [wb] at (4,1) {};
	\node [wb] at (3,1) {};
	\node [wb] at (2,1) {};
	\node [wb] at (1,1) {};
	\node [wb] at (0,1) {};
	\node [wb] at (-1,1) {};
	\node [wb] at (-2,1) {};
	\node [wb] at (-3,1) {};
	\node [wb] at (-4,1) {};
	\node [wb] at (-5,1) {};
	\node [wb] at (-6,1) {};
	\node [wb] at (-7,1) {};
	\node [bb] at (-8,1) {};
	\node [wb] at (-9,1) {};
	
	\node [wb] at (11,2) {};
	\node [wb] at (10,2) {};
	\node [wb] at (9,2) {};
	\node [wb] at (8,2) {};
	\node [wb] at (7,2) {};
	\node [wb] at (6,2) {};
	\node [wb] at (5,2) {};
	\node [wb] at (4,2) {};
	\node [wb] at (3,2) {};
	\node [wb] at (2,2) {};
	\node [wb] at (1,2) {};
	\node [wb] at (0,2) {};
	\node [wb] at (-1,2) {};
	\node [wb] at (-2,2) {};
	\node [wb] at (-3,2) {};
	\node [wb] at (-4,2) {};
	\node [wb] at (-5,2) {};
	\node [wb] at (-6,2) {};
	\node [bb] at (-7,2) {};
	\node [wb] at (-8,2) {};
	\node [bb] at (-9,2) {};
	
	\end{tikzpicture}
\end{center}	
	Thus, the abacus configuration of $Y(\lambda^\circ)$ is		
 
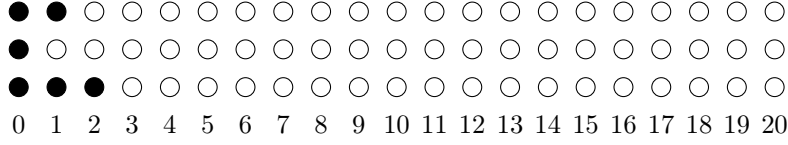
\begin{center}
   \begin{tikzpicture}[scale=0.5, bb/.style={draw,circle,fill,minimum size=2.5mm,inner sep=0pt,outer sep=0pt}, wb/.style={draw,circle,fill=white,minimum size=2.5mm,inner sep=0pt,outer sep=0pt}]
	
	\node [] at (11,-1) {20};
	\node [] at (10,-1) {19};
	\node [] at (9,-1) {18};
	\node [] at (8,-1) {17};
	\node [] at (7,-1) {16};
	\node [] at (6,-1) {15};
	\node [] at (5,-1) {14};
	\node [] at (4,-1) {13};
	\node [] at (3,-1) {12};
	\node [] at (2,-1) {11};
	\node [] at (1,-1) {10};
	\node [] at (0,-1) {9};
	\node [] at (-1,-1) {8};
	\node [] at (-2,-1) {7};
	\node [] at (-3,-1) {6};
	\node [] at (-4,-1) {5};
	\node [] at (-5,-1) {4};
	\node [] at (-6,-1) {3};
	\node [] at (-7,-1) {2};
	\node [] at (-8,-1) {1};
	\node [] at (-9,-1) {0};

	\node [wb] at (11,0) {};
	\node [wb] at (10,0) {};
	\node [wb] at (9,0) {};
	\node [wb] at (8,0) {};
	\node [wb] at (7,0) {};
	\node [wb] at (6,0) {};
	\node [wb] at (5,0) {};
	\node [wb] at (4,0) {};
	\node [wb] at (3,0) {};
	\node [wb] at (2,0) {};
	\node [wb] at (1,0) {};
	\node [wb] at (0,0) {};
	\node [wb] at (-1,0) {};
	\node [wb] at (-2,0) {};
	\node [wb] at (-3,0) {};
	\node [wb] at (-4,0) {};
	\node [wb] at (-5,0) {};
	\node [wb] at (-6,0) {};
	\node [bb] at (-7,0) {};
	\node [bb] at (-8,0) {};
	\node [bb] at (-9,0) {};
	
	\node [wb] at (11,1) {};
	\node [wb] at (10,1) {};
	\node [wb] at (9,1) {};
	\node [wb] at (8,1) {};
	\node [wb] at (7,1) {};
	\node [wb] at (6,1) {};
	\node [wb] at (5,1) {};
	\node [wb] at (4,1) {};
	\node [wb] at (3,1) {};
	\node [wb] at (2,1) {};
	\node [wb] at (1,1) {};
	\node [wb] at (0,1) {};
	\node [wb] at (-1,1) {};
	\node [wb] at (-2,1) {};
	\node [wb] at (-3,1) {};
	\node [wb] at (-4,1) {};
	\node [wb] at (-5,1) {};
	\node [wb] at (-6,1) {};
	\node [wb] at (-7,1) {};
	\node [wb] at (-8,1) {};
	\node [bb] at (-9,1) {};
	
	\node [wb] at (11,2) {};
	\node [wb] at (10,2) {};
	\node [wb] at (9,2) {};
	\node [wb] at (8,2) {};
	\node [wb] at (7,2) {};
	\node [wb] at (6,2) {};
	\node [wb] at (5,2) {};
	\node [wb] at (4,2) {};
	\node [wb] at (3,2) {};
	\node [wb] at (2,2) {};
	\node [wb] at (1,2) {};
	\node [wb] at (0,2) {};
	\node [wb] at (-1,2) {};
	\node [wb] at (-2,2) {};
	\node [wb] at (-3,2) {};
	\node [wb] at (-4,2) {};
	\node [wb] at (-5,2) {};
	\node [wb] at (-6,2) {};
	\node [wb] at (-7,2) {};
	\node [bb] at (-8,2) {};
	\node [bb] at (-9,2) {};	
	\end{tikzpicture}
\end{center}		
 and the corresponding abacus configuration of $X^{1,3}(\lambda^\circ)$ is
   
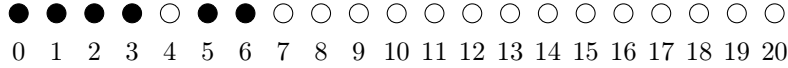
\begin{center}
   \begin{tikzpicture}[scale=0.5, bb/.style={draw,circle,fill,minimum size=2.5mm,inner sep=0pt,outer sep=0pt}, wb/.style={draw,circle,fill=white,minimum size=2.5mm,inner sep=0pt,outer sep=0pt}]
	
	\node [] at (11,-1) {20};
	\node [] at (10,-1) {19};
	\node [] at (9,-1) {18};
	\node [] at (8,-1) {17};
	\node [] at (7,-1) {16};
	\node [] at (6,-1) {15};
	\node [] at (5,-1) {14};
	\node [] at (4,-1) {13};
	\node [] at (3,-1) {12};
	\node [] at (2,-1) {11};
	\node [] at (1,-1) {10};
	\node [] at (0,-1) {9};
	\node [] at (-1,-1) {8};
	\node [] at (-2,-1) {7};
	\node [] at (-3,-1) {6};
	\node [] at (-4,-1) {5};
	\node [] at (-5,-1) {4};
	\node [] at (-6,-1) {3};
	\node [] at (-7,-1) {2};
	\node [] at (-8,-1) {1};
	\node [] at (-9,-1) {0};

	\node [wb] at (11,0) {};
	\node [wb] at (10,0) {};
	\node [wb] at (9,0) {};
	\node [wb] at (8,0) {};
	\node [wb] at (7,0) {};
	\node [wb] at (6,0) {};
	\node [wb] at (5,0) {};
	\node [wb] at (4,0) {};
	\node [wb] at (3,0) {};
	\node [wb] at (2,0) {};
	\node [wb] at (1,0) {};
	\node [wb] at (0,0) {};
	\node [wb] at (-1,0) {};
	\node [wb] at (-2,0) {};
	\node [bb] at (-3,0) {};
	\node [bb] at (-4,0) {};
	\node [wb] at (-5,0) {};
	\node [bb] at (-6,0) {};
	\node [bb] at (-7,0) {};
	\node [bb] at (-8,0) {};
	\node [bb] at (-9,0) {};
	\end{tikzpicture}
	\end{center}
  So $a(\lambda)=(6-4)+2\cdot(4-3)+3\cdot(3-2)+4\cdot(2-1)=11$ and $a(\lambda^\circ)=5-4=1$.

Now, the abacus configuration of
 $X^{3,\tilde{\bf s}} (\ulambda)$ is
  
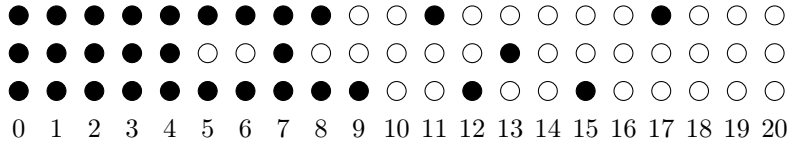
\begin{center}
   \begin{tikzpicture}[scale=0.5, bb/.style={draw,circle,fill,minimum size=2.5mm,inner sep=0pt,outer sep=0pt}, wb/.style={draw,circle,fill=white,minimum size=2.5mm,inner sep=0pt,outer sep=0pt}]
	
	\node [] at (11,-1) {20};
	\node [] at (10,-1) {19};
	\node [] at (9,-1) {18};
	\node [] at (8,-1) {17};
	\node [] at (7,-1) {16};
	\node [] at (6,-1) {15};
	\node [] at (5,-1) {14};
	\node [] at (4,-1) {13};
	\node [] at (3,-1) {12};
	\node [] at (2,-1) {11};
	\node [] at (1,-1) {10};
	\node [] at (0,-1) {9};
	\node [] at (-1,-1) {8};
	\node [] at (-2,-1) {7};
	\node [] at (-3,-1) {6};
	\node [] at (-4,-1) {5};
	\node [] at (-5,-1) {4};
	\node [] at (-6,-1) {3};
	\node [] at (-7,-1) {2};
	\node [] at (-8,-1) {1};
	\node [] at (-9,-1) {0};

	\node [wb] at (11,0) {};
	\node [wb] at (10,0) {};
	\node [wb] at (9,0) {};
	\node [wb] at (8,0) {};
	\node [wb] at (7,0) {};
	\node [bb] at (6,0) {};
	\node [wb] at (5,0) {};
	\node [wb] at (4,0) {};
	\node [bb] at (3,0) {};
	\node [wb] at (2,0) {};
	\node [wb] at (1,0) {};
	\node [bb] at (0,0) {};
	\node [bb] at (-1,0) {};
	\node [bb] at (-2,0) {};
	\node [bb] at (-3,0) {};
	\node [bb] at (-4,0) {};
	\node [bb] at (-5,0) {};
	\node [bb] at (-6,0) {};
	\node [bb] at (-7,0) {};
	\node [bb] at (-8,0) {};
	\node [bb] at (-9,0) {};
	
	\node [wb] at (11,1) {};
	\node [wb] at (10,1) {};
	\node [wb] at (9,1) {};
	\node [wb] at (8,1) {};
	\node [wb] at (7,1) {};
	\node [wb] at (6,1) {};
	\node [wb] at (5,1) {};
	\node [bb] at (4,1) {};
	\node [wb] at (3,1) {};
	\node [wb] at (2,1) {};
	\node [wb] at (1,1) {};
	\node [wb] at (0,1) {};
	\node [wb] at (-1,1) {};
	\node [bb] at (-2,1) {};
	\node [wb] at (-3,1) {};
	\node [wb] at (-4,1) {};
	\node [bb] at (-5,1) {};
	\node [bb] at (-6,1) {};
	\node [bb] at (-7,1) {};
	\node [bb] at (-8,1) {};
	\node [bb] at (-9,1) {};
	
	\node [wb] at (11,2) {};
	\node [wb] at (10,2) {};
	\node [wb] at (9,2) {};
	\node [bb] at (8,2) {};
	\node [wb] at (7,2) {};
	\node [wb] at (6,2) {};
	\node [wb] at (5,2) {};
	\node [wb] at (4,2) {};
	\node [wb] at (3,2) {};
	\node [bb] at (2,2) {};
	\node [wb] at (1,2) {};
	\node [wb] at (0,2) {};
	\node [bb] at (-1,2) {};
	\node [bb] at (-2,2) {};
	\node [bb] at (-3,2) {};
	\node [bb] at (-4,2) {};
	\node [bb] at (-5,2) {};
	\node [bb] at (-6,2) {};
	\node [bb] at (-7,2) {};
	\node [bb] at (-8,2) {};
	\node [bb] at (-9,2) {};
	
	\end{tikzpicture}
\end{center}
and the abacus configuration of
 $X^{3,\tilde{\bf s}} (\boldsymbol{\emptyset})$ is
  
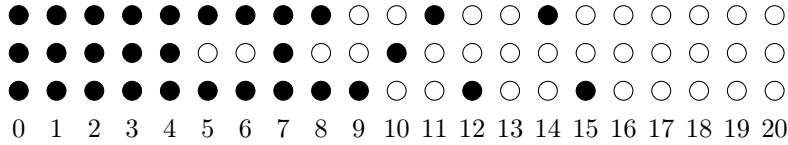
\begin{center}
   \begin{tikzpicture}[scale=0.5, bb/.style={draw,circle,fill,minimum size=2.5mm,inner sep=0pt,outer sep=0pt}, wb/.style={draw,circle,fill=white,minimum size=2.5mm,inner sep=0pt,outer sep=0pt}]
	
	\node [] at (11,-1) {20};
	\node [] at (10,-1) {19};
	\node [] at (9,-1) {18};
	\node [] at (8,-1) {17};
	\node [] at (7,-1) {16};
	\node [] at (6,-1) {15};
	\node [] at (5,-1) {14};
	\node [] at (4,-1) {13};
	\node [] at (3,-1) {12};
	\node [] at (2,-1) {11};
	\node [] at (1,-1) {10};
	\node [] at (0,-1) {9};
	\node [] at (-1,-1) {8};
	\node [] at (-2,-1) {7};
	\node [] at (-3,-1) {6};
	\node [] at (-4,-1) {5};
	\node [] at (-5,-1) {4};
	\node [] at (-6,-1) {3};
	\node [] at (-7,-1) {2};
	\node [] at (-8,-1) {1};
	\node [] at (-9,-1) {0};

	\node [wb] at (11,0) {};
	\node [wb] at (10,0) {};
	\node [wb] at (9,0) {};
	\node [wb] at (8,0) {};
	\node [wb] at (7,0) {};
	\node [bb] at (6,0) {};
	\node [wb] at (5,0) {};
	\node [wb] at (4,0) {};
	\node [bb] at (3,0) {};
	\node [wb] at (2,0) {};
	\node [wb] at (1,0) {};
	\node [bb] at (0,0) {};
	\node [bb] at (-1,0) {};
	\node [bb] at (-2,0) {};
	\node [bb] at (-3,0) {};
	\node [bb] at (-4,0) {};
	\node [bb] at (-5,0) {};
	\node [bb] at (-6,0) {};
	\node [bb] at (-7,0) {};
	\node [bb] at (-8,0) {};
	\node [bb] at (-9,0) {};
	
	\node [wb] at (11,1) {};
	\node [wb] at (10,1) {};
	\node [wb] at (9,1) {};
	\node [wb] at (8,1) {};
	\node [wb] at (7,1) {};
	\node [wb] at (6,1) {};
	\node [wb] at (5,1) {};
	\node [wb] at (4,1) {};
	\node [wb] at (3,1) {};
	\node [wb] at (2,1) {};
	\node [bb] at (1,1) {};
	\node [wb] at (0,1) {};
	\node [wb] at (-1,1) {};
	\node [bb] at (-2,1) {};
	\node [wb] at (-3,1) {};
	\node [wb] at (-4,1) {};
	\node [bb] at (-5,1) {};
	\node [bb] at (-6,1) {};
	\node [bb] at (-7,1) {};
	\node [bb] at (-8,1) {};
	\node [bb] at (-9,1) {};
	
	\node [wb] at (11,2) {};
	\node [wb] at (10,2) {};
	\node [wb] at (9,2) {};
	\node [wb] at (8,2) {};
	\node [wb] at (7,2) {};
	\node [wb] at (6,2) {};
	\node [bb] at (5,2) {};
	\node [wb] at (4,2) {};
	\node [wb] at (3,2) {};
	\node [bb] at (2,2) {};
	\node [wb] at (1,2) {};
	\node [wb] at (0,2) {};
	\node [bb] at (-1,2) {};
	\node [bb] at (-2,2) {};
	\node [bb] at (-3,2) {};
	\node [bb] at (-4,2) {};
	\node [bb] at (-5,2) {};
	\node [bb] at (-6,2) {};
	\node [bb] at (-7,2) {};
	\node [bb] at (-8,2) {};
	\node [bb] at (-9,2) {};
	
	\end{tikzpicture}
\end{center}
whence
 $a_{\tilde{\bf s},e}(\ulambda)=(15-14)+2\cdot(13-12)+3\cdot(12-11)+4\cdot(11-10)=10$.
   \end{exa}

    \section{Schur elements of Ariki--Koike algebras}
    
    \subsection{Generic Ariki-Koike algebras}
Let ${\bf q}:=(Q_1,\,\ldots,\,Q_{l}\,;\,q)$ be a set of $l+1$ indeterminates  and
set $R:=\mathbb{Z}[{\bf q},{\bf q}^{-1}]$.
The {\it Ariki-Koike algebra} \index{$\mathcal{H}^{\bf q}_{n}$} $\mathcal{H}^{\bf q}_{n}$  is the associative $R$-algebra (with unit) with generators $T_0,\,T_1,\,\ldots,\,T_{n-1}$ and relations:
\begin{center}
$\begin{array}{rl}
(T_0 -Q_1) (T_0 -Q_2)\cdots(T_0 -Q_{l})=0& \\  \smallbreak
(T_i-q)(T_i+1)=0  & \text{for $1\leq i \leq n-1$}\\  \smallbreak
T_0T_1T_0T_1=T_1T_0T_1T_0&\\  \smallbreak
T_iT_{i+1}T_i=T_{i+1}T_iT_{i+1} & \text{for $1\leq i \leq n-2$}\\  \smallbreak
T_iT_j=T_jT_i  &\text{for  $0\leq i <j \leq n-1$ with $j-i>1$.}
\end{array}$
\end{center}

It follows from \cite{ArKo} and Ariki's semisimplicity criterion \cite{Arsem} that the algebra $\Q({\bf q})\mathcal{H}^{\bf q}_{n}$ is split semisimple. We have a bijection $\Pi^l_n \longrightarrow \Irr(\Q({\bf q})\mathcal{H}^{\bf q}_{n}),\, \ulambda  \longmapsto \chi^{\ulambda}$, where $\Pi^l_n$ denotes the set of $l$-partitions of $n$ and $ \Irr(\Q({\bf q})\mathcal{H}^{\bf q}_{n})$ denotes the set of irreducible characters of  $\Q({\bf q})\mathcal{H}^{\bf q}_{n}$.

There exists a canonical symmetrising trace \index{$\tau$}$\tau$ on $\mathcal{H}^{\bf q}_{n}$ in the sense of Brou\'e--Malle--Michel \cite{BMM}. This trace was defined by Bremke and Malle \cite{BreMa} over a field, and then it was shown to be a symmetrising trace over $R$ by Malle and Mathas in \cite{MaMa}\footnote{The extra condition needed for the trace to be canonical is supposed to be settled by \cite[Theorem 5.2]{GIM}.}. We have
$$ \tau = \sum_{\ulambda \in \Pi^l_n} \frac{1}{s_\ulambda} \chi^{\ulambda},$$
where \index{$s_\ulambda $} $s_\ulambda \in R$ is the \emph{Schur element} of \index{$\chi^\ulambda $} $\chi^\ulambda$ with respect to $\tau$.

Two independent descriptions of the Schur elements of $\mathcal{H}^{\bf q}_{n}$  have been given by Geck--Iancu--Malle  \cite{GIM} and Mathas \cite{Mat}. In both articles, the Schur elements are given as fractions in $\Q({\bf q})$. However, since the Schur elements belong to the Laurent polynomial ring $R$, we know that the denominator  always divides the numerator. In \cite{CJ1} we have given a cancellation-free formula for these Schur elements, that is, we have explicitly described their irreducible factors in $R$. This formula uses the CJ-hook lengths and, following our work in \cite{CJ2}, it can be read as follows:

\begin{Th}\label{canfreeform} Let $\ulambda=(\lambda^{1},\lambda^{2},\ldots,\lambda^{l}) \in \Pi^l_n$, and set  $X:=X^{1,{\boldsymbol{0}}}(\ulambda)$.   The Schur element $s_{\ulambda}$ is given by
$$s_{\ulambda}=(-1)^{n(l-1)}q^{-N(\bar{\lambda})}(q-1)^{-n}  \prod_{(a,b,i,j) \in \mathcal{H}^{CJ}(X)}  (q^{a-b}Q_iQ_j^{-1}-1)$$
where $\bar{\lambda}$ is the partition of $n$ obtained by reordering all the numbers in $\ulambda$
and $N(\bar{\lambda}):=\sum_{i \geq 1}  (i-1)\bar{\lambda}_i.$
\end{Th}

Note that the term $(q-1)^{-n}$ can be cancelled-out by the terms corresponding to the classical hook lengths $(a,b,i,i)$.
Following  Corollary \ref{bij},  for $X:=X^{1,{\boldsymbol{0}}}(\ulambda)$, we have the same number of 
  elements in  $\mathcal{H}^{BGO}(X)$ and $ \mathcal{H}^{CJ}(X)$.  In addition, by Proposition \ref{inj}, 
   the expressions  $\prod_{(a,b,i,j)\in\mathcal{H}^{BGO}(X)}  (q^{a-b}Q_iQ_j^{-1}-1)$ and $\prod_{(a,b,i,j)\in \mathcal{H}^{CJ}(X)}  (q^{a-b}Q_iQ_j^{-1}-1)$
    are equal up to multiplication by products of invertible elements in $R$.  We can thus rewrite the above formula in terms of BGO-hook lengths:

\begin{Th}\label{canfreeformBGO} Let $\ulambda=(\lambda^{1},\lambda^{2},\ldots,\lambda^{l}) \in \Pi^l_n$, and set  $X:=X^{1,{\boldsymbol{0}}}(\ulambda)$.   Then there exists  $u_\ulambda \in R^\times$ such that
$$s_{\ulambda}=u_\ulambda (q-1)^{-n}  \prod_{(a,b,i,j) \in \mathcal{H}^{BGO}(X)}  (q^{a-b}Q_iQ_j^{-1}-1).$$
\end{Th}
From now on, we set $\tilde{s}_\ulambda:=(q-1)^n s_\ulambda$.

\subsection{Specialised Ariki--Koike algebras}

Let  $L$ be a field and let $\theta : R \to L$ 
   be a specialisation of $R$.  Set $\xi_i:=\theta (Q_i)$ for $1\leq i \leq l$,  $u:=\theta (q)$ and ${\bf u}:=(\xi_1,\,\ldots,\,\xi_{l}\,;\,u)$. 
 Assume that the algebra  $L\mathcal{H}^{\bf u}_{n}$ is split.
 By some general results on symmetric algebras (\cite[Theorem 7.5.11]{GePf}, \cite[Proposition 4.4]{GR})
 we have that the block of $L\mathcal{H}^{\bf u}_{n}$ containing $\chi^\ulambda$ is a $1 \times 1$ identity matrix if and only if $\theta(s_\ulambda) \neq 0$. In particular,
  the algebra $L\mathcal{H}^{\bf u}_{n}$ is semisimple if and only if $\theta(s_\ulambda) \neq 0$ for all $\ulambda \in \Pi_n^l$  (\cite[Theorem 7.4.7]{GePf}).  In \cite[Theorem 4.2]{CJ1}, we have shown that this criterion
  in combination with the form of the Schur elements given by Theorem \ref{canfreeform} allows us to recover Ariki's semisimplicity criterion for Ariki--Koike algebras, which is the following \cite[Main Theorem]{Arsem}:
 
 \begin{Th}\label{Ariki's semisimplicity} The algebra
 $L \mathcal{H}^{\bf u}_n$ is  semisimple if and only if 
 $$\prod_{1\leq i \leq n}(1+u+\cdots+u^{i-1}) \prod_{1 \leq a <b \leq l}\,\,\prod_{-n<h<n}(u^h\xi_a-\xi_b) \neq 0.$$
 \end{Th}

In any case, we have a well-defined decomposition matrix \index{$D_\theta$} 
$D_\theta=([V^{\ulambda}:M])_{\ulambda\in \Pi^l_n,\,M\in \text{Irr} ( L  \mathcal{H}_n^{\bf u})}$.   There is a useful result 
 by Dipper and Mathas \cite{DiMa} which allows us to restrict ourselves to a very specific situation in order to study $D_\theta$. In order to do this, we set $\mathcal{U}:=\left\{1,\ldots ,{l}\right\}$ and we assume that we have a partition 
$$\mathcal{U}=\mathcal{U}_1 \sqcup   \mathcal{U}_2 \sqcup \ldots  \sqcup \mathcal{U}_t $$
which is the finest with respect to the property
\begin{equation}\label{theproperty}
\prod_{1\leq \alpha<\beta\leq t}\,\, \prod_{(a,b)\in \mathcal{U}_{\alpha}\times \mathcal{U}_{\beta}}\,\, \prod_{-n<h<n} (u^h \xi_a-\xi_b) \neq 0.
\end{equation}
 For $i=1,\ldots,t$, write $\mathcal{U}_i:=\{a_{i,1},\ldots,a_{i,m_i}\}$ with $a_{i,1} < \cdots < a_{i,m_i}$. Whenever ${\bf f}=(f_1,\ldots,f_{l})$ is a sequence  indexed by $\mathcal{U}$, we will write ${\bf f}[i]$ for the sequence $(f_{a_{i,1}},\ldots,f_{a_{i,m_i}})$.

 \begin{Th}\label{Nmorita}
{\bf (The  Morita equivalence of Dipper and Mathas)}  For $i=1,\ldots,t$, we set ${\bf u_i}:=((\xi_1,\,\ldots,\,\xi_{l})[i];\,u)$. 
The algebra $L\mathcal{H}^{\bf u}_{n}$ is Morita equivalent to the algebra
$$\bigoplus_{n_1,\ldots ,n_t\geq 0 \atop n_1+\ldots +n_t=n}  L \mathcal{H}^{\bf u_1}_{n_1}\otimes_L   L\mathcal{H}^{\bf u_2}_{n_2}\otimes_L  \ldots \otimes_L    L\mathcal{H}_{n_t}^{\bf u_t}.    $$
\end{Th}

\begin{Rem}
Recently, Rostam \cite{Salim} has produced an explicit isomorphism between $L\mathcal{H}^{\bf u}_{n}$ and
$$\bigoplus_{n_1,\ldots ,n_t\geq 0 \atop n_1+\ldots +n_t=n} {\rm Mat}_{\frac{n!}{n_1!\dots n_t!}} \left(L \mathcal{H}^{\bf u_1}_{n_1}\otimes_L   L\mathcal{H}^{\bf u_2}_{n_2}\otimes_L  \ldots \otimes_L    L\mathcal{H}_{n_t}^{\bf u_t}\right),    $$
which implies the Morita equivalence of Dipper and Mathas.
\end{Rem}

Therefore, to study the representation theory of Ariki-Koike algebras (in particular the blocks, the decomposition numbers etc.), we can restrict ourselves to the study of specialisations of the form 
 where $\xi_i=u^{s_i}$ for some $s_i \in \N$ (we do not need negative powers, because we can always add the same integer to all the $s_i$'s and the algebra does not change). This also covers the case of cyclotomic Ariki--Koike algebras, as studied in \cite[\S 3.4]{CJ2}. However, note that without the restriction imposed by property \eqref{theproperty}, the algebra $L\mathcal{H}^{\bf u}_{n}$ is not necessarily semisimple  (whereas cyclotomic Ariki--Koike algebras always are).

More generally, let $k \in \N$ and ${\bf s} \in \mathbb{N}^l$. Let $\theta_{k,{\bf s}}:R \rightarrow \Q(q)$ be a specialisation such that $\theta_{k,{\bf s}}(q)=q^k$ and $\theta_{k,{\bf s}}(Q_i)=q^{s_i}$ for all $i=1,\ldots,l$. 
Set $\mathfrak{q}:=\{q^{s_1},\ldots,q^{s_l};q^k\}$. 
Then, for all $\ulambda \in \Pi_n^l$,
\begin{equation}
\theta_{k,{\bf s}}(\tilde{s}_{\ulambda})=(-1)^{n(l-1)}q^{-kN(\bar{\lambda})}  \prod_{h \in \mathcal{HL}^{CJ}_{k,{\bf s}} (X^{1,\boldsymbol{0}}(\ulambda))}  (q^h-1).
\end{equation}
Therefore, the algebra $\Q(q)\mathcal{H}^{\mathfrak{q}}_{n}$ is semisimple if and only if
$0 \notin \mathcal{HL}^{CJ}_{k,{\bf s}} (X^{1,\boldsymbol{0}}(\ulambda))$ for all $\ulambda \in \Pi_n^l$.
 More specifically, by \cite[Theorem 7.2.6]{GePf}, the simple module of character $\chi^\ulambda$ remains irreducible and projective after specialisation if and only if $0 \notin \mathcal{HL}^{CJ}_{k,{\bf s}} (X^{1,\boldsymbol{0}}(\ulambda))$. Moreover, it is known (see \cite[5.5.3]{GJ}) that the valuation of the above Laurent polynomial is equal to $-a_{{\bf s},k}(\ulambda)$ (see \S \ref{afonc}). We thus have
 \begin{equation}\label{theta-schur}
 \theta_{k,{\bf s}}(\tilde{s}_{\ulambda})= \pm \, q^{-a_{{\bf s},k}(\ulambda)} \prod_{h \in |\mathcal{HL}^{CJ}_{k,{\bf s}} (X^{1,\boldsymbol{0}}(\ulambda))|}  (q^h-1)
 \end{equation}
where the sign is equal to $(-1)^{n(l-1)}$ times the number of negative elements in $\mathcal{HL}^{CJ}_{k,{\bf s}} (X^{1,\boldsymbol{0}}(\ulambda))$.

If now we take $q$ to be a primitive $e$-th root of unity, it follows from \cite{JL} that the definition of $(e,{\bf s})$-core naturally generalises the following two facts from type $A$ to all other types:\smallbreak
\begin{itemize}
\item the block containing $\chi^\ulambda$ is a $1 \times 1$ identity matrix if and only $\ulambda$ is an $(e,{\bf s})$-core \cite[Corollary 4.1]{JL};\smallbreak
\item two characters $\chi^\ulambda$ and $\chi^\umu$ are in the same block if and only if $\ulambda$ and $\umu$ have the same  $(e,{\bf s})$-core \cite[Corollary 4.4]{JL}.
\end{itemize}

\subsection{Factorisation of Schur elements in type $A$}
    One of the main results of this paper is the following formula that makes a connection between the Schur elements of a partition, its core and its quotient. It is derived directly from Formula \eqref{theta-schur}, using Corollary \ref{particore} and Proposition \ref{a-function}.
    
    \begin{Th}\label{mainTypeA}
Let $e,n \in \mathbb{Z}_{>0}$ and let $\lambda$ be a partition of $n$. 
Let $\lambda^\circ$ be the $e$-core of $\lambda$ and $\ulambda$ its $e$-quotient.
Let ${\bf s}=(s_1,\ldots,s_e)$ be the multicharge of the $e$-symbol of $\lambda$ and $\tilde{\bf s}=(es_1,es_2+1,\ldots,es_e+(e-1))$. Then
$$\tilde{s}_\lambda = \pm \,\tilde{s}_{\lambda^\circ} \cdot \theta_{e,\tilde{\bf s}}(\tilde{s}_{\ulambda}).$$
    \end{Th}
    
From the above theorem, one can easily deduce several block-theoretic connections between the partition $\lambda$, its $e$-core $\lambda^\circ$ and its $e$-quotient $\ulambda$. For example, if $\lambda$ is an $e'$-core for some $e' \neq e$, then $\lambda^\circ$ is also an $e'$-core (this was first proved in \cite{Ol} for $e'$ coprime to $e$, and the general case was proved in \cite{GN}) and $\ulambda$ is an $(e',\tilde{\bf s})$-core.
    
\subsection{Divisibility of Schur elements in other types} For $l > 1$, we have Proposition \ref{prop2.3}, which implies the following at the level of Schur elements:

\begin{Prop}\label{unmaintheorem} Let $e,l,n \in \mathbb{Z}_{>0}$ and ${\bf s}=(s_1,\ldots,s_l) \in \mathbb{N}^l$.
Let $\ulambda$ be an $l$-partition and let $\ulambda^\circ$ be the $(e,{\bf s})$-core of $\ulambda$.
Then $\theta_{1,{\bf s}^\circ}(\tilde{s}_{\ulambda^\circ})$ divides 
\begin{equation}\label{lastelement}
\theta_{1,{\bf s}}(\tilde{s}_{\ulambda})\,\, \cdot \prod_{1 \leq i <j \leq l} \prod_{h=1}^{|s_i-s_j|-1} (q^h-1).
\end{equation}
\end{Prop}

Unfortunately, we often have $\theta_{1,{\bf s}}(\tilde{s}_{\ulambda})=0$. In fact, we have $\theta_{1,{\bf s}}(\tilde{s}_{\ulambda}) \neq 0$ if and only if $\ulambda$ is an $(\infty, {\bf s})$-core. For example, for ${\bf s}=0$, this is equivalent to having $\lambda^1=\lambda^2=\cdots=\lambda^l$. Nevertheless, if $\ulambda$ is an $(\infty, {\bf s})$-core, the above proposition implies that, for any $e \in \mathbb{Z}_{>0}$, the Schur element of the $e$-core of $\ulambda$ divides the Schur element of $\ulambda$ (modulo the right-hand term of \eqref{lastelement}).

\printindex

\end{document}